\documentclass[11pt,a4paper]{article}

\usepackage[english]{babel}
\usepackage[utf8]{inputenc}
\usepackage[T1]{fontenc}
\usepackage{lmodern}

\usepackage{graphicx}
\graphicspath{{assets/}}
\usepackage{graphbox}
\usepackage{amssymb}
\usepackage{amsmath}
\usepackage{amsthm}
\usepackage{mathtools}
\usepackage{stmaryrd}
\usepackage[colorlinks=true,linkcolor=blue,citecolor=blue,urlcolor=blue]{hyperref}
\usepackage{cleveref}

\usepackage{csquotes,xpatch}
\usepackage[style=numeric,giveninits=true,sorting=nyt,maxbibnames=99,backend=biber]{biblatex}
\DeclareFieldFormat{titlecase}{\MakeSentenceCase*{#1}}

\usepackage[a4paper,margin=1in]{geometry}

\usepackage{enumitem}
\usepackage{microtype}
\usepackage{epigraph}
\input{macros.sty}

\newtheorem{theorem}{Theorem}[section]
\newtheorem{lemma}[theorem]{Lemma}
\newtheorem{corollary}[theorem]{Corollary}

\newtheorem{conjecture}[theorem]{Conjecture}

\theoremstyle{definition}
\newtheorem{definition}[theorem]{Definition}

\theoremstyle{remark}
\newtheorem{remark}[theorem]{Remark}

\title{Local flag algebras}
\author{Eoin Davey \and Eoin Hurley \and R\'emi de Joannis de Verclos \and Ross J. Kang \and Jan Volec}
\date{14 July 2026}

\begin{document}

\maketitle

\begin{abstract}
We introduce \emph{local flag algebras}, a variant of Razborov's flag
algebra framework in which densities are normalised by the maximum
degree $\Delta(G)$ rather than the order $|G|$. The framework supports
the same semidefinite-method machinery as the classical version, but is tailored to extremal problems that scale with the
maximum degree. As an illustrative first application we bound the number of pentagons
in a triangle-free graph $G$ as a function of $|G|$ and
$\Delta(G)$.
\end{abstract}


\epigraph{\footnotesize\ttfamily Beware of bugs in the above code; I have only proved it correct, not tried it.}{---Donald Knuth\footnotemark}
\footnotetext{\url{https://staff.fnwi.uva.nl/p.vanemdeboas/knuthnote.pdf}}

\section{Introduction}
\label{sec:introduction}

We introduce \emph{local flag algebras}, a variant of Razborov's flag
algebra framework~\cite{razborovFlagAlgebras2007} in which densities are
normalised by the maximum degree $\Delta(G)$ rather than the order
$|G|$. Classical flag algebras compute with limits of densities
normalised by $\binom{|G|}{k}$, and so are suited to extremal problems
that scale with the order; local flag algebras instead use \emph{local
densities} normalised by $\binom{\Delta(G)}{k}$, capturing problems that
scale with the maximum degree. The variant retains the whole
semidefinite-method apparatus of the classical theory --- a flag
product, limit functionals, an averaging (unlabelling) operator, a
positivity cone, and weak duality --- and adds a transfer principle that
lifts asymptotic flag-algebra inequalities to unconditional bounds for
all graphs. We develop it from first principles in
Sections~\ref{sec:classical-recap}--\ref{sec:localflags}.

These local densities resemble neighbourhood statistics of a random vertex,
in the spirit of the local-convergence theory of Benjamini and
Schramm~\cite{benjaminiSchrammRecurrence2001}. However, the parallel is only
partial: that theory typically fixes a uniform bound on the maximum degree, whereas
sampling a fixed number $k$ of vertices at a bounded depth from a random vertex 
and normalising by $\Delta(G)^k$ allows $\Delta(G) \to \infty$, which is closer to the theory of dense limits due to
Lov\'asz and Szegedy~\cite{lovaszSzegedyLimits2006}.

The framework is
the main contribution of this paper. To demonstrate the method we
use it to make a first attack on one specific extremal problem --- a bounded-degree refinement of
Erd\H{o}s's pentagon problem --- one which we have formulated specifically for this purpose.
With a suitable reduction, this problem can also be treated within the classical
framework of flag algebras, so we offer it here more as an illustration of local flag algebras
than as a problem that demands it; the framework's distinctive value
appears where the local normalisation is essential, as in the companion
paper's~\cite{daveySECLocalFlags2024} bounds on the strong chromatic
index. In Section~\ref{sec:bruhn-joos}, we give a separate application as a preview of that direction.

Let $P(G)$ denote the number of unordered pentagons (induced copies of
$C_5$) in a graph $G$. In 1983 Erd\H{o}s~\cite{erdosProblemsGraphTheory1984}
conjectured that for every triangle-free graph $G$,
\[
P(G) \leq (|G|/5)^5,
\]
with equality on the balanced blowup of $C_5$;
Grzesik~\cite{grzesikMaximumNumberFivecycles2012} and Hatami, Hladk\'y,
Kr\'a\v{l}, Norine, and Razborov~\cite{hatamiNumberPentagonsTrianglefree2013}
independently proved it, both using Razborov's flag algebras. We propose the
following bounded-degree analogue: how many pentagons can a triangle-free graph
contain in terms of \emph{both} its order $|G|$ and its maximum degree
$\Delta(G)$? Since $\Delta(G) \leq |G|-1$, we have
$|G|\,\Delta(G)^4 \leq |G|^5$, the two agreeing up to a constant factor
exactly when $\Delta(G) = \Theta(|G|)$ --- as on Erd\H{o}s's extremal
blowup of $C_5$, where $\Delta(G) = 2|G|/5$. A bound of the form
$P(G) \leq c\,|G|\,\Delta(G)^4$ is thus a genuine sharpening for sparse
graphs, smaller by a factor $(|G|/\Delta(G))^4$.

Local flag algebras yield two such bounds. The first serves as a warm-up.

\begin{theorem}[simple pentagon bound]
\label{thm:simple}
For every triangle-free graph $G$,
\[
P(G) \leq \frac{|G|\,\Delta(G)^4}{40}.
\]
\end{theorem}

\begin{theorem}[tighter pentagon bound]
\label{thm:tight}
For every triangle-free graph $G$,
\[
P(G) \leq 0.02073\cdot |G|\,\Delta(G)^4.
\]
\end{theorem}

Counting pentagons one vertex at a time cannot beat the constant
$1/40$ of Theorem~\ref{thm:simple}: a single vertex can lie on close
to $\Delta(G)^4/8$ pentagons (Lemma~\ref{lem:per-vertex-tight}).
Theorem~\ref{thm:tight} surpasses this ceiling by also counting the
pentagons through the neighbours of each vertex.

In the classical (unbounded-degree) setting the maximizer is the
blowup of $C_5$, so one might expect it to remain extremal in the
bounded-degree normalisation as well. It does not: the $C_5$-blowup
attains only $P(G)/(|G|\,\Delta(G)^4) = 1/80 = 0.0125$, and several
triangle-free graphs exceed it --- for instance the Petersen graph,
whose $12$ pentagons give ratio $12/810 = 2/135 \approx 0.0148$ (it is the
extremal example at $\Delta = 3$; Theorem~\ref{thm:small-degree}).
Even this is not the best construction we are aware of, and we conjecture
that it is the Clebsch graph that attains the maximum.

\begin{conjecture}[bounded-degree pentagon, sharp form]
\label{conj:pentagon}
For every triangle-free graph $G$,
\[
P(G) \leq \frac{12}{625}\cdot |G|\,\Delta(G)^4.
\]
Moreover, if $G$ is connected then the bound is sharp on the Clebsch graph
$\mathrm{Cl} \cong \mathrm{SRG}(16,5,0,2)$~\cite{brouwerDistanceRegularGraphs1989}
and its balanced blowups.
\end{conjecture}

With
$|\mathrm{Cl}| = 16$, $\Delta(\mathrm{Cl}) = 5$, and $P(\mathrm{Cl}) = 192$,
the ratio for the Clebsch graph is $192/(16\cdot 5^4) = 12/625 = 0.0192$.
Theorem~\ref{thm:tight}'s constant $0.02073$ is within
$\approx 1.08\times$ of $12/625$.

\begin{lemma}[Clebsch-blowup tightness]
\label{lem:clebsch}
For every $k \geq 1$, the $k$-blowup $\mathrm{Cl}[k]$ of the
Clebsch graph $\mathrm{Cl}$ is triangle-free with $|\mathrm{Cl}[k]| = 16k$,
$\Delta(\mathrm{Cl}[k]) = 5k$, and $P(\mathrm{Cl}[k]) = 192 k^5$, so
\[
\frac{P(\mathrm{Cl}[k])}{|\mathrm{Cl}[k]|\,\Delta(\mathrm{Cl}[k])^4}
= \frac{12}{625}.
\]
\end{lemma}

At maximum degree five Conjecture~\ref{conj:pentagon}, together with its extremal
characterisation, is a theorem.

\begin{theorem}[Clebsch characterisation at $\Delta = 5$]
\label{thm:delta5}
Every triangle-free graph $G$ with $\Delta(G) \leq 5$ satisfies
\[
P(G) \leq 12\,|G|,
\]
with equality if and only if every component of $G$ is isomorphic to
$\mathrm{Cl}$. In particular, Conjecture~\ref{conj:pentagon} holds for
every triangle-free graph of maximum degree exactly five, and its
maximisers at maximum degree exactly five are the disjoint unions of
copies of $\mathrm{Cl}$.
\end{theorem}

The proof is elementary and independent of the semidefinite method;
the exact constant at $\Delta = 5$ lies past the plateau of
the size-$8$ certificate (Section~\ref{sec:vandermonde-floor}). For the
maximum degrees $\Delta \geq 6$ that remain open, a computer
search finds no ratio $P(G)/(|G|\Delta(G)^4)$ above $12/625$;
Section~\ref{sec:clebsch} and the closing note
\textit{\nameref{sec:lean-and-empirical}} give the details.
The same elementary per-vertex count yields analogous bounds at maximum
degrees three and four, sharp at three.

\begin{theorem}[small maximum degree]
\label{thm:small-degree}
Let $G$ be triangle-free.
\begin{enumerate}[label=(\roman*)]
\item If $\Delta(G) \leq 3$, then $P(G) \leq \tfrac{6}{5}\,|G|$, with
equality if and only if every component of $G$ is a Petersen graph.
\item If $\Delta(G) \leq 4$, then $P(G) \leq \tfrac{24}{5}\,|G|$.
\end{enumerate}
Together with Theorem~\ref{thm:delta5} and the trivial case
$\Delta(G) \leq 2$ (where $P(G) \leq |G|/5$), Conjecture~\ref{conj:pentagon}
holds for every triangle-free graph with $\Delta(G) \leq 5$; if $G$ has
no isolated vertices, equality holds if and only if every component of
$G$ is isomorphic to $\mathrm{Cl}$.
\end{theorem}

Unlike for part~(i), the bound in part~(ii)
is not known to be sharp. The densest triangle-free graph of maximum
degree four we know, the circulant $C_{12}(2,3)$ on $\mathbb{Z}_{12}$
with $i \sim i \pm 2, i \pm 3$, has $48$ pentagons and ratio
$1/64 = 0.015625$, whereas the per-vertex count caps the ratio at
$3/160 = 0.01875$; the true maximum at $\Delta = 4$ lies between these
and is open.

We outline the proofs. That of Theorem~\ref{thm:simple} uses a
doubling argument and a regularity reduction to reach an asymptotic bound on a 4-path
count in a 2-coloured auxiliary class. We express this count as a
limit functional $\phi(O)$ on a size-$5$ objective $O$ in the local
flag algebra, and bound it by $\phi(O) \leq 1/4$ through an explicit
semidefinite-programming (SDP) certificate written out as a convex combination of
extension-difference constraints, a black-vertex normalisation, and
two Cauchy--Schwarz blocks. The proof of Theorem~\ref{thm:tight}
follows the same architecture for a richer per-vertex functional $Q$,
built from the number $P(G,v)$ of pentagons through a vertex $v$ together
with the corresponding counts at its neighbours (defined precisely in
Section~\ref{sec:tight}). The size-8 SDP
delivers $\phi(O_Q) \leq 0.4146$ through an explicit rational
certificate, verified block-by-block by $LDL^\top$ decompositions together
with an arithmetic identity for the convex combination.
Lemma~\ref{lem:clebsch} is a trace computation on the spectrum of
$\mathrm{Cl}$ with the standard $k$-blowup expansion, and
Theorems~\ref{thm:delta5} and~\ref{thm:small-degree} bypass the
semidefinite method entirely: pentagons through a vertex are counted by
a degree-tuned weight function on the non-neighbours, with an equality
analysis pinning down the extremal graphs (Section~\ref{sec:clebsch}
and Appendix~\ref{app:small-degree}).

Beyond the pentagon illustration developed here, the framework is meant for
reuse: any extremal problem that scales with the maximum degree is a
candidate. At the end of this paper, we describe an additional illustrative application. In the companion paper~\cite{daveySECLocalFlags2024} we pursue this application 
--- upper bounds on the strong chromatic index of
graphs and bipartite graphs --- much further. We expect other applications to follow.

\paragraph{Note on AI and Lean.}

A desired standard in the flag algebra community is, for robustness and increased confidence, to generate the needed SDP in two different software implementations.
Here our approach is new: we formally verify in Lean~4 the proof that our flag-algebra implementation produces, rather than relying on a separate codebase for corroboration.
We also wish to disclose that we obtained Theorems~\ref{thm:delta5} and~\ref{thm:small-degree}, intended as auxiliary supporting results toward Conjecture~\ref{conj:pentagon}, by deploying a commercially available agentic AI system to construct their proofs directly in Lean~4 under our guidance.
See the closing notes \textit{\nameref{sec:lean-and-empirical}} and \textit{\nameref{sec:AI-usage}} for additional details.

\paragraph{Organisation.}

Sections~\ref{sec:classical-recap} and~\ref{sec:localflags} develop the
framework: the former recalls just enough
of classical flag algebras to fix notation, and the latter builds the
local flag algebra framework with all framework lemmas proved.
Sections~\ref{sec:simple}--\ref{sec:clebsch} carry out the illustrative
pentagon application.
Section~\ref{sec:simple} proves Theorem~\ref{thm:simple} via the
size-$5$ certificate. Section~\ref{sec:tight} proves
Theorem~\ref{thm:tight} via the size-$8$ certificate.
Section~\ref{sec:certificates} describes the size-$5$ and size-$8$
certificates as mathematical objects.
Section~\ref{sec:clebsch} proves Lemma~\ref{lem:clebsch} and
Theorem~\ref{thm:delta5}, and surveys the empirical evidence
supporting Conjecture~\ref{conj:pentagon}.
Section~\ref{sec:bruhn-joos} closes with a second, briefer illustration of the local flag algebra method,
recovering the Bruhn--Joos sparsity bound.
Appendix~\ref{app:small-degree} proves Theorem~\ref{thm:small-degree}.

\section{Classical flag algebras: brief recap}
\label{sec:classical-recap}

We collect just enough of Razborov's classical flag algebra
machinery~\cite{razborovFlagAlgebras2007} to motivate the local
variant. The reader already familiar enough with this should skip to
Section~\ref{sec:localflags}.

A \emph{type} $\sigma$ is a labelled graph on the vertex set
$[|\sigma|]$. Given a graph class $\Gcl$, a $\sigma$-flag $(F,\theta)$
consists of a graph $F \in \Gcl$ together with an injection
$\theta\colon[|\sigma|] \to V(F)$ realising $\sigma$ as the induced
subgraph $F[\im\theta]$ on the image $\im\theta$ of $\theta$. We write
$\Gcl^\sigma_n$ for the set of
isomorphism classes of $\sigma$-flags of size $n$ and $\Gcl^\sigma$ for
the union over $n$.

For $\sigma$-flags $F, H$, the \emph{induced count} $c(F;H)$ is the
number of subsets $\im\theta_H \subseteq U \subseteq V(H)$ --- where
$\theta_H$ is the type embedding of $H$ --- with $H[U]
\cong F$ as $\sigma$-flags. The \emph{induced density} is
$p(F;H) = c(F;H)/\binom{|H|-|\sigma|}{|F|-|\sigma|}$. Razborov's chain
rule states that for $n \geq \max(|F|,|F'|)$ and $H \in \Gcl^\sigma_n$,
\begin{equation}
\label{eq:chain-rule}
p(F;H)\,p(F';H) = p(F \cdot F'; H) + O(1/|H|),
\end{equation}
where $F \cdot F' := \sum_{H \in \Gcl^\sigma_{|F|+|F'|-|\sigma|}} p(F,F';H)\,H$,
and $p(F,F';H)$ is
the probability that a uniformly random partition of the unlabelled
vertices of $H$ into parts of sizes $|F|-|\sigma|$ and $|F'|-|\sigma|$
induces $\sigma$-flags isomorphic to $F$ and $F'$ respectively. This
product makes $\R\Gcl^\sigma$, modulo the chain-rule relations, into a
commutative associative unital algebra $\Acl^\sigma$.

A \emph{limit functional} $\phi \in \Phi^\sigma$ is an algebra
homomorphism $\Acl^\sigma \to \R$ arising as
$\phi(F) = \lim_k p(F;G_k)$ for a convergent sequence
$(G_k) \subseteq \Gcl^\sigma$. The \emph{semantic cone}
$\SemCone^\sigma \subseteq \Acl^\sigma$ consists of those $f$ with
$\phi(f) \geq 0$ for all $\phi$. The \emph{averaging operator}
$\llbracket\,\cdot\,\rrbracket\colon \Acl^\sigma \to \Acl^\emptyset$ is
$\llbracket F \rrbracket = q_\sigma(F)\,\downflag{F}$, where
$\downflag{F}$ forgets the labelling and
$q_\sigma(F) = |\mathrm{Stab}(\theta)|/|F|!$, where $\mathrm{Stab}(\theta)$
counts the orderings of $V(F)$ that induce the flag $(F, \theta)$. It
preserves positivity:
$\llbracket f^2 \rrbracket \in \SemCone^\emptyset$ for every
$f \in \Acl^\sigma$. The semidefinite method produces upper bounds
$\phi(f) \leq \lambda$ by exhibiting a decomposition
\[
\lambda\emptyset - f = \sum_i \alpha_i g_i + \sum_j \llbracket h_j^2 \rrbracket
\]
with $\alpha_i \geq 0$ and $g_i \in \SemCone^\emptyset$ known.

The classical framework does not bear directly on our target. The
quantity $P(G,v)/\Delta(G)^4$ is not of the form $p(F;G)$ for any
single classical flag $F$: the natural denominator is $\Delta(G)^4$,
not $|G|^4$, and densities in $\Delta(G)$ are not subdensities in
$|G|$. The next section introduces the local variant.

\section{The local flag algebra framework}
\label{sec:localflags}

Throughout this section we fix a graph class $\Gcl$ (not necessarily
hereditary) and a graph parameter $\Delta\colon \Gcl \to \N_0$. In all
our applications $\Delta$ is the maximum-degree function. We write
$\HeredG$ for the hereditary closure of $\Gcl$.

\subsection{Local densities and local flags}

\begin{definition}[local density]
\label{def:rho}
For $(F,\theta) \in \HeredG^\sigma$ and $(G,\eta) \in \Gcl^\sigma$,
\[
\rho\bigl((F,\theta);(G,\eta)\bigr)
:= \frac{c\bigl((F,\theta);(G,\eta)\bigr)}{\binom{\Delta(G)}{|F|-|\sigma|}}.
\]
We write $\rho(F;G)$ when the embeddings are clear from context.
\end{definition}

In contrast to $p(F;G)$, the quantity $\rho(F;G)$ is not a probability
and may be unbounded. For example, $\rho(\vertex; G) = |G|/\Delta(G)$
is unbounded over the class of all graphs. On the other hand
$\rho(\edgemarked; G) = \deg(\eta(1))/\Delta(G) \leq 1$.

\begin{definition}[local $\sigma$-flag]
\label{def:local-flag}
A $\sigma$-flag $(F,\theta) \in \HeredG^\sigma$ is a \emph{local
$\sigma$-flag} if
\begin{enumerate}[label=(\roman*)]
\item the map $G \mapsto \rho(F;G)$ is bounded on $\Gcl^\sigma$, and
\item every label extension $F^v$ obtained by labelling an unlabelled
vertex $v \in V(F) \setminus \im\theta$ is also a local
$\sigma'$-flag, where $\sigma'$ extends $\sigma$ at $v$.
\end{enumerate}
We write $\Gloc^\sigma$ for the set of local $\sigma$-flags and
$\Glocsub{n}^\sigma$ for those of size $n$.
\end{definition}

Condition (ii) is essential: bounded density alone does not propagate
under the algebra operations we shall introduce.

\begin{lemma}
\label{lem:second-cond-needed}
There exist a class $\Gcl$ and a $\sigma$-flag $F$ such that
$\rho(F;\cdot)$ is bounded but $F$ is not a local $\sigma$-flag.
\end{lemma}

\begin{proof}
Let $\Gcl$ be the class of three-vertex-coloured graphs (black, red,
blue) with exactly one red vertex, $\Delta(G)^2$ blue vertices, and no
edge between red and blue. Let $F = \redbluenonedge$ be the
$\emptyset$-flag with one red vertex and one blue vertex and no edge.
For every $G \in \Gcl$, $c(F;G) = \Delta(G)^2$, so
$\rho(F;G) = 2\Delta(G)/(\Delta(G)-1) \leq 4$ is bounded. The label extension $F' = \redbluenonedgemarked$ of type
$\redvertex$ has $c(F';G) = \Delta(G)^2$ and $\rho(F';G) = \Delta(G)$,
which is unbounded. Hence $F$ fails condition (ii) of
Definition~\ref{def:local-flag}.
\end{proof}

The intuition is that $F$ is a local $\sigma$-flag exactly when
$\Delta(G)$ bounds the choices in any embedding of $F$'s unlabelled
vertices. For the maximum-degree parameter, this holds whenever each
connected component of $F$ contains a labelled anchor.

\begin{lemma}
\label{lem:component-local}
If $\Delta$ is the maximum-degree function, every $\sigma$-flag whose
connected components each contain a labelled vertex is a local
$\sigma$-flag.
\end{lemma}

\begin{proof}
Each unlabelled vertex of $F$ has finite distance to a labelled vertex
in its component. Process unlabelled vertices in order of increasing
distance, extending the embedding one vertex at a time. Each step has
at most $\Delta(G)$ choices, since the new vertex is adjacent to a
previously embedded one. Hence
$c(F;G) \in O(\Delta(G)^{|F|-|\sigma|})$, so $\rho(F;G)$ is bounded.
The same argument applied to a label extension proves condition (ii).
\end{proof}

In the pentagon application of Section~\ref{sec:simple} we shall need a
strengthening that admits a second class of anchor: a black vertex in a
2-coloured class with a controlled black set. We prove this strengthening
(Lemma~\ref{lem:pentagon-local}) where it is used, in
Section~\ref{sec:simple}.

\subsection{The local flag product}

\begin{definition}[local flag product]
\label{def:product}
For $F, F' \in \Gloc^\sigma$ and $n = |F|+|F'|-|\sigma|$, set
\[
F \cdot F' := \sum_{H \in \Glocsub{n}^\sigma} p(F,F';H)\,H,
\]
extending bilinearly to $\R\Gloc^\sigma$ to obtain an algebra
$\Lcl^\sigma$.
\end{definition}

The sum is over local flags. The following theorem ensures it is
equivalent to summing over the entire hereditary closure.

\begin{theorem}
\label{thm:product-in-loc}
Let $F, F' \in \Gloc^\sigma$ and $H \in \HeredG^\sigma_n$ with $n =
|F|+|F'|-|\sigma|$. If $p(F,F';H) > 0$, then $H \in \Gloc^\sigma$.
\end{theorem}

\begin{proof}
Let $\theta, \theta', \eta$ be the $\sigma$-embeddings of $F, F', H$.
Since $p(F,F';H) > 0$, there exist subsets $U,U' \subseteq V(H)$ with
$U \cap U' = \im\eta$, $H[U] \cong F$, $H[U'] \cong F'$, and (by
$|U|+|U'|-|\sigma| = n = |V(H)|$) $U \cup U' = V(H)$.

\emph{Bounded density.} Let $(G,\zeta) \in \Gcl^\sigma$. Every
$\im\zeta \subseteq V \subseteq V(G)$ with $H \cong G[V]$ induces
embeddings of $U, U'$ into $V(G)$ whose images intersect in $\im\zeta$
and induce copies of $F, F'$. Each $V$ therefore yields a pair (copy
of $F$, copy of $F'$); distinct $V$ yield distinct pairs because
$U \cup U' = V(H)$ recovers $V$ from the pair. Hence
\[
c(H;G) \leq c(F;G)\cdot c(F';G)
\in O\bigl(\Delta(G)^{(|F|-|\sigma|)+(|F'|-|\sigma|)}\bigr)
= O\bigl(\Delta(G)^{|H|-|\sigma|}\bigr).
\]

\emph{Label-extension.} Pick an unlabelled $v \in V(H)$. Since
$U \cup U' = V(H)$ and the two parts play symmetric roles, we may assume
$v \in U \setminus \im\eta$. Each copy of $H^v$ in $G$
corresponds to a pair (copy of $F^v$, copy of $F'$) where $F^v$ is the
label extension of $F$ at the vertex identified with $v$. By locality
of $F$, $c(F^v;\cdot) \in O(\Delta(G)^{|F|-|\sigma|-1})$, so
$c(H^v;\cdot) \in O(\Delta(G)^{|H|-|\sigma|-1})$.
\end{proof}

\begin{corollary}
\label{cor:product-all-flags}
For $F, F' \in \Gloc^\sigma$ and $n = |F|+|F'|-|\sigma|$,
\[
F \cdot F'
= \sum_{H \in \Glocsub{n}^\sigma} p(F,F';H)\,H
= \sum_{H \in \HeredG^\sigma_n} p(F,F';H)\,H.
\]
\end{corollary}

The product makes $\Lcl^\sigma$ an algebra.

\begin{lemma}
\label{lem:local-assoc}
The algebra $\Lcl^\sigma$ is commutative, associative, and unital with
unit $\sigma$.
\end{lemma}

\begin{proof}
Commutativity is immediate. For associativity, expand
$(F_1 \cdot F_2) \cdot F_3$ using
Corollary~\ref{cor:product-all-flags}. The classical chain
rule~\eqref{eq:chain-rule} applies on the $\HeredG^\sigma$-sum and
produces an expression symmetric in $F_1, F_2, F_3$. For the unit,
$c(\sigma;G) = 1$ so $\rho(\sigma;G) = 1$, and $F \cdot \sigma$ has the
unique nonzero summand $F$.
\end{proof}

\subsection{The product limit}

\begin{theorem}[local product limit]
\label{thm:product-limit}
For $f, g \in \Lcl^\sigma$ and $G \in \Gcl^\sigma$,
\[
\rho(f;G)\,\rho(g;G) = \rho(f \cdot g; G) + O(1/\Delta(G)),
\]
where the implicit constant depends on $f$ and $g$ but not on $G$. In
particular, every local limit functional $\phi \in \Phi^\sigma$ --- the
notation of the classical recap, now denoting a limit of local
densities $\rho(\cdot;G)$ along a $\Delta$-increasing sequence (one
with $\Delta(G_k) \to \infty$; Lemma~\ref{lem:limit-exists}) --- is an
algebra homomorphism
$\Lcl^\sigma \to \R$.
\end{theorem}

\begin{proof}
By bilinearity reduce to $f = F$, $g = F'$. Set $k = |\sigma|$ and
$n = |F|+|F'|-k$. Compute
\[
\rho(F;G)\,\rho(F';G) - \rho(F \cdot F';G)
= \frac{c(F;G)\,c(F';G)}{\binom{\Delta(G)}{|F|-k}\binom{\Delta(G)}{|F'|-k}}
- \frac{\sum_H c(F,F';H)\,c(H;G)}{\binom{n-k}{|F|-k}\binom{\Delta(G)}{n-k}}.
\]
The identity
$\binom{\Delta}{n-k}\binom{n-k}{|F|-k}
 = \binom{\Delta}{|F|-k}\binom{\Delta-(|F|-k)}{|F'|-k}$
together with
$\binom{\Delta-(|F|-k)}{|F'|-k}/\binom{\Delta}{|F'|-k} = 1 + O(1/\Delta)$
shows the two denominators agree up to a $1 + O(1/\Delta)$ factor;
both are $\Theta(\Delta(G)^{n-k})$. It suffices to bound the numerator
discrepancy
\begin{equation}
\label{eq:overlap}
\Bigl|c(F;G)\,c(F';G) - \sum_{H \in \HeredG^\sigma_n} c(F,F';H)\,c(H;G)\Bigr|
\in O\bigl(\Delta(G)^{n-k-1}\bigr).
\end{equation}
The left product counts pairs $(U,U')$ with
$\im\eta \subseteq U,U' \subseteq V(G)$, $G[U] \cong F$, $G[U'] \cong F'$,
with no constraint on $U \cap U'$. By
Corollary~\ref{cor:product-all-flags} the right sum counts the same
pairs subject to $U \cap U' = \im\eta$. The difference counts pairs
with $U \cap U' \supsetneq \im\eta$.

Fix such a pair and pick an unlabelled $v \in V(F)$ whose image $x$ lies
in $U \cap U'\setminus\im\eta$. The vertex $x$ also corresponds to an
unlabelled $w \in V(F')$. By locality of $F'$ applied to $(F')^w$ (of
type $\sigma$ extended at $w$),
$c((F')^w; G^{x}) \in O(\Delta(G)^{|F'|-k-1})$. Summing over the
$O(\Delta(G)^{|F|-k})$ choices of $U$ and the constant number of
$(v,w)$ positions yields~\eqref{eq:overlap}.

For the multiplicativity of $\phi$, divide both sides by the
denominators, take $G = G_m$ along the defining $\Delta$-increasing
sequence, and pass to the limit; the $O(1/\Delta)$ term vanishes.
\end{proof}

\subsection{Positivity and local types}

The \emph{semantic cone} is
\[
\SemCone^\sigma := \{f \in \Lcl^\sigma : \phi(f) \geq 0
\text{ for all } \phi \in \Phi^\sigma\}.
\]
The averaging operator $\llbracket\,\cdot\,\rrbracket\colon
\R\Gcl^\sigma \to \R\Gcl^\emptyset$ follows the classical
definition: $\llbracket F \rrbracket = q_\sigma(F)\,\downflag{F}$. In
the local setting, restriction to $\Lcl^\sigma$ does not automatically
land in $\Lcl^\emptyset$ unless $\sigma$ itself is a \emph{local type}.

\begin{definition}[local type]
\label{def:local-type}
A type $\sigma$ is a \emph{local type} if $\downflag{F} \in
\Gloc^\emptyset$ for every $F \in \Gloc^\sigma$.
\end{definition}

\begin{lemma}
\label{lem:local-type-equiv}
A type $\sigma$ is local if and only if $\downflag{\sigma}$ is a local
$\emptyset$-flag.
\end{lemma}

\begin{proof}
The forward direction is immediate from $\sigma \in \Gloc^\sigma$.

For the reverse direction, suppose $\downflag{\sigma} \in
\Gloc^\emptyset$ and let $F \in \Gloc^\sigma$. We bound
$c(\downflag{F};H)$ for $H \in \Gcl^\emptyset$. Every embedding $g$ of $F$ into $H$ induces a
$\sigma$-embedding $\theta_g := g \circ \theta$ of $\sigma$ into $H$,
where $\theta$ is the $\sigma$-embedding of $F$. Group the count of
embeddings by $\theta_g$:
\begin{equation}
\label{eq:fibre}
c(\downflag{F};H) = \sum_{\psi\colon\sigma \to H}
\bigl|\{g \colon F \to H,\ g \circ \theta = \psi\}\bigr|.
\end{equation}
The outer sum has at most $c(\sigma;H) \in O(\Delta(H)^{|\sigma|})$
terms by hypothesis. Each inner fibre is bounded by the number of
extensions of $\psi$ to a $\sigma$-embedding of $F$, which by
locality of $F$ (condition (i)) is
$c(F; (H,\psi)) \in O(\Delta(H)^{|F|-|\sigma|})$. Hence
$c(\downflag{F};H) \in O(\Delta(H)^{|F|})$, so $\rho(\downflag{F};\cdot)$
is bounded. The label-extension condition for $\downflag{F}$ follows by
the same argument applied to $\downflag{F^v}$ for an unlabelled
$v \in V(F)$.
\end{proof}

\begin{remark}[generalised bounded density]
\label{rem:gen-bounded-density}
Lemma~\ref{lem:local-type-equiv} extends to the following statement:
if $F \in \Gloc^\sigma$ and $\downflag{\sigma} \in \Gloc^\emptyset$,
then for every type $\tau$ and every $\tau$-labelling $\tilde F$ of
the underlying graph $\downflag{F}$ (so $\downflag{\tilde F} =
\downflag{F}$), the map $H \mapsto \rho(\tilde F;H)$ is bounded on
$\Gcl^\tau$. The proof refines the fibre decomposition of
Lemma~\ref{lem:local-type-equiv} by organising
$\tau$-embeddings according to the overlap
$j := |\im\theta \cap \im\eta_\tau|$ between the $\sigma$-image and
the $\tau$-image; each overlap class admits two bounds that combine
via Vandermonde's inequality
$\binom{\Delta}{a}\binom{\Delta}{b} \leq \binom{a+b}{a}^2 \cdot
\binom{\Delta}{a+b}$ (for $\Delta \geq a+b$, which holds for the large
$\Delta(H)$ at issue since $a, b \leq |F|$)
to give a single $\binom{\Delta(H)}{|F|-|\tau|}$. The constant
$\binom{|F|-|\tau|}{|\sigma|-j}^2 \leq 4^{|F|}$ absorbs the overlap
index sum. We do not need this strengthening for the pentagon
applications, which use the shorter Lemma~\ref{lem:pentagon-local}.
\end{remark}

When $\sigma$ is a local type, the averaging operator restricts to
$\llbracket\,\cdot\,\rrbracket\colon \Lcl^\sigma \to \Lcl^\emptyset$.

\subsection{Existence of limit functionals}

\begin{lemma}[existence of limit functionals]
\label{lem:limit-exists}
Let $\Gcl$ be a graph class, $\sigma$ a type, and $(G_k)_{k \in \N}
\subset \Gcl^\sigma$ a \emph{$\Delta$-increasing sequence}, that is,
one with $\Delta(G_k) \to \infty$. There exists a
subsequence $(G_{n_k})$ and a limit functional
$\phi \in \Phi^\sigma$ with
$\phi(F) = \lim_k \rho(F;G_{n_k})$ for every local $\sigma$-flag $F$.
The functional satisfies $\phi(F) \geq 0$ for every local flag,
$\phi(\sigma) = 1$, respects $\sigma$-flag isomorphism, vanishes on
non-local flags, and is an algebra homomorphism on $\Lcl^\sigma$.
\end{lemma}

\begin{proof}
Enumerate isomorphism classes of $\sigma$-flags as $(\mathrm{cls}_j)$;
this is a countable union of finite sets. For each local class
$\mathrm{cls}_j$, Definition~\ref{def:local-flag}(i) supplies a uniform
bound $B_j := \sup_k \rho(\mathrm{cls}_j; G_k) < \infty$. For non-local
classes set $B_j = 0$ and $u_k(\mathrm{cls}_j) = 0$. Otherwise define
$u_k(\mathrm{cls}_j) := \rho(\mathrm{cls}_j; G_k)$. The product space
$X := \prod_j [0,B_j]$ is compact by Tychonoff's theorem and
metrisable since the index set is countable, hence sequentially
compact. Pass to a subsequence $(n_k)$ along which $u_{n_k} \to a \in
X$ pointwise.

Define $\phi(F) := a(\mathrm{cls}(F))$ for $F$ local and $\phi(F) := 0$
otherwise, extending linearly to $\Lcl^\sigma$. Non-negativity,
$\phi(\sigma) = \lim_k \rho(\sigma;G_{n_k}) = 1$, and respect of
$\sigma$-flag isomorphism are immediate. Multiplicativity follows from
Theorem~\ref{thm:product-limit}: for $v, w$ supported on local
classes,
\[
\rho(v;G_{n_k})\,\rho(w;G_{n_k}) - \rho(v \cdot w; G_{n_k}) \to 0,
\]
and passing to the limit gives $\phi(v)\phi(w) = \phi(v \cdot w)$.
\end{proof}

\subsection{Asymptotic averaging and positivity preservation}

\begin{lemma}[asymptotic averaging]
\label{lem:asymp-averaging}
For $F \in \Gloc^\sigma$ and $G \in \Gcl$ with $\rho(\llbracket\sigma
\rrbracket;G) > 0$,
\[
\E_\theta\bigl[\rho(F;(G,\theta))\bigr]
= \frac{\rho(\llbracket F \rrbracket; G)}{\rho(\llbracket \sigma \rrbracket;G)}
\cdot \bigl(1 + O(1/\Delta(G))\bigr),
\]
where $\theta$ ranges uniformly over $\sigma$-embeddings into $G$.
\end{lemma}

\begin{proof}
The induced--local conversion gives, for any $\sigma$-flag $H$ of size
$|H|$,
\[
\rho(H;G)
= \frac{\binom{|G|-|\sigma|}{|H|-|\sigma|}}{\binom{\Delta(G)}{|H|-|\sigma|}}
\cdot p(H;G).
\]
The first factor depends on $H$ only through $|H|-|\sigma|$. Classical
averaging~\cite[Lem.~1.18]{razborovFlagAlgebras2007} gives
$\E_\theta[p(F;(G,\theta))] = p(\llbracket F\rrbracket;G)/
p(\llbracket\sigma\rrbracket;G)$. Substituting and simplifying via
$\binom{|G|}{|F|}\binom{|F|}{|\sigma|} = \binom{|G|}{|\sigma|}\binom{|G|-|\sigma|}{|F|-|\sigma|}$
and the analogous identity in $\Delta(G)$,
\[
\frac{\rho(\llbracket F\rrbracket;G)}{\rho(\llbracket\sigma\rrbracket;G)}
= \E_\theta[\rho(F;(G,\theta))]
\cdot \frac{\binom{\Delta(G)}{|F|-|\sigma|}}{\binom{\Delta(G)-|\sigma|}{|F|-|\sigma|}}.
\]
The ratio on the right is $1 + O(1/\Delta(G))$.
\end{proof}

\begin{lemma}[positivity preservation]
\label{lem:positivity}
Let $\sigma$ be a local type. Then $\llbracket\SemCone^\sigma\rrbracket
\subseteq \SemCone^\emptyset$. In particular,
$\llbracket f^2 \rrbracket \in \SemCone^\emptyset$ for every $f \in
\Lcl^\sigma$.
\end{lemma}

\begin{proof}
Suppose $f \in \SemCone^\sigma$ and $\llbracket f \rrbracket \notin
\SemCone^\emptyset$. Some $\phi \in \Phi^\emptyset$ has
$\phi(\llbracket f \rrbracket) < 0$. By Lemma~\ref{lem:limit-exists},
choose a $\Delta$-increasing sequence $(G_k)$ for which $\phi$ is the
limit functional. By Lemma~\ref{lem:asymp-averaging},
\[
\rho(\llbracket f \rrbracket; G_k)
= \rho(\llbracket \sigma \rrbracket; G_k)
\cdot \E_{\theta_k}\bigl[\rho(f;(G_k,\theta_k))\bigr]
\cdot (1 + o(1)).
\]
Locality of $\sigma$ gives
$\rho(\llbracket\sigma\rrbracket;G_k)$ bounded, so the expectation is
eventually negative. For large $k$ pick $\theta_k$ minimising
$\rho(f;(G_k,\theta_k))$. Apply Lemma~\ref{lem:limit-exists} to the
sequence $(G_k,\theta_k)$ to obtain $\phi' \in \Phi^\sigma$ with
$\phi'(f) < 0$, contradicting $f \in \SemCone^\sigma$.
\end{proof}

\subsection{The semidefinite method and weak duality}
\label{sec:sdp-method}

To bound $\phi(f)$ for a target $f \in \Lcl^\emptyset$, search for a
decomposition
\begin{equation}
\label{eq:sdp-decomp}
\lambda\emptyset - f
= \sum_i \alpha_i\,g_i + \sum_j \llbracket h_j^2 \rrbracket
\end{equation}
with $\alpha_i \geq 0$, $g_i \in \SemCone^\emptyset$ known, and
$h_j \in \Lcl^{\sigma_j}$ for local types $\sigma_j$. Any such
decomposition yields $\phi(f) \leq \lambda$ for every $\phi \in
\Phi^\emptyset$. A semidefinite-program solver finds the
decomposition; the verification is a finite arithmetic check on
integer or rational data.

Fix a size $n$ and an unlabelled flag basis
$\Glocsub{n}^\emptyset = \{F_1, \dots, F_m\}$. Write $f = \sum_i c_i
F_i$. Each Cauchy--Schwarz summand $\llbracket h^2 \rrbracket$ with
$h = \sum_p z_p\,\mathrm{basis}^\sigma_p$ expands as
\[
\llbracket h^2 \rrbracket
= \sum_{i=1}^m \trace\bigl(M_\sigma^{(i)} \cdot Y\bigr)\,F_i,
\qquad Y = z z^\top,
\]
where $M_\sigma^{(i)}$ is the integer structure matrix obtained from
the chain-rule expansion of
$\mathrm{basis}^\sigma_p \cdot \mathrm{basis}^\sigma_q$ followed by
unlabelling. The bound~\eqref{eq:sdp-decomp} is the primal of a
semidefinite program with positive semidefinite (PSD) variables $Y_\sigma$ and non-negative
scalars $\alpha_i$; the dual has decision variables $x_i \approx
\phi(F_i)$ and a single scalar normalisation $x_{i^\star} = 1$ for the
index $i^\star$ of the empty flag.

\begin{lemma}[weak duality]
\label{lem:weak-duality}
For any feasible primal $(Y_\sigma,\alpha,\lambda)$ and feasible dual
$x$,
\[
\lambda \geq \sum_i c_i\,x_i.
\]
In particular, if the primal SDP attains value $\lambda^\star$, then
$\phi(f) \leq \lambda^\star$ for every $\phi \in \Phi^\emptyset$.
\end{lemma}

\begin{proof}
Multiply the primal equality $c_i = -\sum_\sigma \trace(M_\sigma^{(i)}
Y_\sigma) - \sum_j \alpha_j A_j^{(i)} + \lambda\,[F_i = \emptyset]$ by
$-x_i$ and sum over $i$:
\begin{align*}
-\sum_i c_i x_i
&= \sum_\sigma \trace\Bigl(\bigl(\textstyle\sum_i x_i M_\sigma^{(i)}\bigr) Y_\sigma\Bigr)
+ \sum_j \alpha_j \bigl(\textstyle\sum_i x_i A_j^{(i)}\bigr) - \lambda.
\end{align*}
The trace term is non-negative because $\sum_i x_i M_\sigma^{(i)}
\succeq 0$ (dual feasibility) and $Y_\sigma \succeq 0$ (primal
feasibility) and the trace of a product of PSD matrices is
non-negative. The $\alpha$ term is non-negative by primal and dual
feasibility. The $-\lambda$ term uses $x_{i^\star} = 1$. Rearranging
gives the bound.

For the limit-functional consequence, set $x_i = \phi(F_i)$. Dual
feasibility of $x = \phi$ encodes
$\phi(\llbracket h^2 \rrbracket) \geq 0$ (which is
Lemma~\ref{lem:positivity}) together with $\phi(g_j) \geq 0$ and
$\phi(\emptyset) = 1$.
\end{proof}

\subsection{Regular classes and extension constraints}
\label{sec:extension-trick}

When $\Gcl$ consists of regular graphs the framework acquires a generic
family of cone elements. For a type $\sigma$ of size $k$ and an index
$i \in [k]$, define the \emph{extension} $\ext_i^\sigma$ to be the
sum of all $\sigma$-flags $F \in \HeredG^\sigma$ of size $k+1$ with an
edge between the unique unlabelled vertex and the vertex labelled $i$.

By Lemma~\ref{lem:component-local} each summand is a local
$\sigma$-flag, so $\ext_i^\sigma \in \Lcl^\sigma$.

\begin{lemma}
\label{lem:ext-eq-one}
If $\Gcl$ consists of regular graphs, then $\phi(\ext_i^\sigma) = 1$
for every $\phi \in \Phi^\sigma$ and $i \in [|\sigma|]$.
\end{lemma}

\begin{proof}
For $(G,\eta) \in \Gcl^\sigma$, summing the count over all $F$ in the
support of $\ext_i^\sigma$ counts the choices of an unlabelled vertex
$u$ adjacent to $\eta(i)$ in $V(G)\setminus\im\eta$, which is
$\deg(\eta(i)) - |\{j : \eta(j) \in N(\eta(i))\}| = \Delta(G) - O(1)$.
Dividing by $\binom{\Delta(G)}{1} = \Delta(G)$ gives $1 - o(1)$.
\end{proof}

\begin{corollary}
\label{cor:ext-diff}
For every type $\sigma$, $i, j \in [|\sigma|]$, $f \in \Lcl^\sigma$,
\[
\phi(\ext_i^\sigma - \ext_j^\sigma) = 0,
\qquad
\phi(f \cdot \ext_i^\sigma) = \phi(f).
\]
In particular, $\ext_i^\sigma - \ext_j^\sigma$,
$f \cdot \ext_i^\sigma - f$, and $f - f \cdot \ext_i^\sigma$ all lie
in $\SemCone^\sigma$.
\end{corollary}

\begin{corollary}[unlabelled extension]
\label{cor:unlabel-ext}
If $\sigma$ is a local type, then for $\phi \in \Phi^\emptyset$,
$f \in \Lcl^\sigma$, and $i, j \in [|\sigma|]$,
\[
\phi\bigl(\llbracket \ext_i^\sigma - \ext_j^\sigma \rrbracket\bigr) = 0,
\qquad
\phi\bigl(\llbracket f \cdot \ext_i^\sigma \rrbracket\bigr)
= \phi\bigl(\llbracket f \rrbracket\bigr).
\]
\end{corollary}

\begin{proof}
By Corollary~\ref{cor:ext-diff} both $\ext_i^\sigma - \ext_j^\sigma$
and its negative lie in $\SemCone^\sigma$. Lemma~\ref{lem:positivity}
applied to each gives that $\llbracket\ext_i^\sigma -
\ext_j^\sigma\rrbracket$ and its negative lie in $\SemCone^\emptyset$,
hence both are zero on $\phi$. The second identity follows from
$f - f \cdot \ext_i^\sigma \in \SemCone^\sigma$ and its negative.
\end{proof}

The extension relations are doubly useful: they generate a family of
cone elements at no further cost, and they let us lift an element from
$\Lcl^\sigma$ to a span over larger flags without changing its
$\phi$-value.

\section{A simple pentagon bound}
\label{sec:simple}

Sections~\ref{sec:simple}--\ref{sec:clebsch} illustrate the framework of
Section~\ref{sec:localflags} on the bounded-degree pentagon problem,
beginning with the simpler of the two bounds. We prove
Theorem~\ref{thm:simple} in three moves: a reduction of
$P(G,v)$ (the pentagons through a single vertex) to a
black-red-red-black 4-path count in a 2-coloured auxiliary class, an
expression of that count as $\phi(O)$ for a size-$5$ objective $O$ in
the local flag algebra, and a size-$5$ SDP bound $\phi(O) \leq 1/4$.

\subsection{Reductions}

\begin{lemma}[asymptotic suffices]
\label{lem:asymp-suffices}
If $P(G) \lesssim \lambda |G| \Delta(G)^4$ as $\Delta(G) \to \infty$
over triangle-free $G$, then $P(G) \leq \lambda |G|\Delta(G)^4$ for
every triangle-free $G$.
\end{lemma}

\begin{proof}
Suppose $G_0$ violates the bound with
$\rho_0 := P(G_0)/(|G_0|\Delta(G_0)^4) > \lambda$. Construct $G_{i+1}$
by replacing each vertex $v$ with two copies $v_0, v_1$ and joining
$\{u_a, v_b\}$ for every $a, b \in \{0,1\}$ whenever $uv \in E(G_i)$.
Then $|G_{i+1}| = 2|G_i|$ and $\Delta(G_{i+1}) = 2\Delta(G_i)$.
The construction preserves triangle-freeness: a triangle in $G_{i+1}$
projects to one in $G_i$. Each pentagon in $G_i$ lifts to $2^5$
pentagons in $G_{i+1}$,
so $P(G_{i+1}) \geq 2^5 P(G_i)$ and
$P(G_i)/(|G_i|\Delta(G_i)^4) \geq \rho_0$ for every $i$, contradicting
the asymptotic bound along the sequence $(G_i)$.
\end{proof}

\begin{lemma}[regular suffices]
\label{lem:regular-suffices}
For every triangle-free $G$ there exists a triangle-free regular $G'$
with $\Delta(G') = \Delta(G)$ and
$P(G')/(|G'|\Delta(G')^4) \geq P(G)/(|G|\Delta(G)^4)$.
\end{lemma}

\begin{proof}
Iteratively form $G_{i+1}$ as two disjoint copies of $G_i$ together
with an edge between the two copies of every vertex $v$ satisfying
$\deg_{G_i}(v) < \Delta(G_i)$. Adding edges only at non-maximum-degree
vertices preserves $\Delta(G_{i+1}) = \Delta(G_i)$. Each vertex gains at most one cross-copy neighbour (its own image), so
the added edges form a matching between the copies and create no
triangle. From $|G_{i+1}| = 2|G_i|$ and $P(G_{i+1}) \geq 2 P(G_i)$
the ratio is non-decreasing. After at most
$\Delta(G_0) - \delta(G_0) \leq \Delta(G_0)$ iterations the minimum
degree reaches $\Delta(G_0)$.
\end{proof}

\begin{lemma}[per-vertex count suffices]
\label{lem:local-count}
Let $P(G,v) := |\{C_5 \subseteq G : v \in V(C_5)\}|$. If
$P(G,v)/\Delta(G)^4 \lesssim \lambda$ as $\Delta(G) \to \infty$, then
$P(G)/(|G|\Delta(G)^4) \lesssim \lambda/5$.
\end{lemma}

\begin{proof}
Each pentagon contains five vertices and is counted at five of them,
so $\sum_{v} P(G,v) = 5 P(G)$. Apply the asymptotic bound termwise.
\end{proof}

\subsection{A reduction to a coloured 4-path count}

Let $\Gcl$ denote the class of $\{R,B\}$-vertex-coloured graphs that
are triangle-free, regular, in which the black set is independent and
of size exactly $\Delta(G)$. The hereditary closure $\HeredG$ drops the
regularity requirement.

\begin{lemma}[black-red-red-black suffices]
\label{lem:brrb-suffices}
Let $\brrb$ denote the black-red-red-black 4-path, viewed as an
$\emptyset$-flag of size $4$ in $\Gcl$. If
$c(\brrb;G)/\Delta(G)^4 \lesssim \lambda$ as $\Delta(G) \to \infty$
over $\Gcl$, then $P(H,v)/\Delta(H)^4 \lesssim \lambda$ over regular
triangle-free $H$.
\end{lemma}

\begin{proof}
Let $H$ be regular triangle-free and $v \in V(H)$. Any pentagon
through $v$ uses $v$, two vertices of $N(v)$, and two vertices of
$V(H)\setminus(N(v)\cup\{v\})$, in that cyclic order. Colour $H$ by
declaring $N(v)$ black and the remaining vertices (other than $v$)
red, and delete $v$ from $H$. The resulting coloured graph $H'$ is
triangle-free with $N(v)$ an independent black set of size $\Delta(H)$;
its red vertices have degree $\Delta(H)$ and its black vertices degree
$\Delta(H) - 1$, so $H' \in \HeredG$ with $\Delta(H') = \Delta(H)$.
Restoring regularity at the black vertices changes the $4$-path count
by $O(\Delta(H)^3)$, so we may take $H' \in \Gcl$. The pentagons
through $v$ then correspond to black-red-red-black $4$-paths in $H'$,
up to an $O(\Delta(H)^3)$ boundary correction for pentagons that re-use
vertices at distance $1$ from $v$:
$P(H,v) = c(\brrb;H') + O(\Delta(H)^3)$.
\end{proof}

The next lemma extends Lemma~\ref{lem:component-local} to admit black
vertices as anchors.

\begin{lemma}
\label{lem:pentagon-local}
A $\sigma$-flag $F \in \HeredG^\sigma$ is a local $\sigma$-flag if and
only if every connected component of $F$ contains a labelled vertex or
a black vertex.
\end{lemma}

\begin{proof}
($\Leftarrow$) Induction on $|F| - |\sigma|$. The base case
$|F| = |\sigma|$ gives $c(F;G) = 1$. For the inductive step, every
component has an anchor (labelled or black); for each unlabelled
vertex $u$, let $d(u)$ be its distance to an anchor in its component.
Pick an unlabelled $v$ maximising $d(v)$ and set $F' = F[V(F)
\setminus \{v\}]$; the labelled-or-black anchor property persists in
$F'$ (a vertex cut off from its anchor by deleting $v$ would lie farther
from that anchor than $v$, contradicting the maximality of $d(v)$).
By induction $c(F';G) \in O(\Delta(G)^{|F|-|\sigma|-1})$. To count
extensions to $F$, fix a copy of $F'$ in $G$; the vertex $u$ playing
the role of $v$ satisfies one of:
\begin{itemize}
\item $v$ was black in $F$. Then $u$ is black in $G$; since the black
set has size exactly $\Delta(G)$ there are $\Delta(G)$ choices.
\item $v$ was red in $F$. Then by maximality of $d(v)$, $v$ has at
least one neighbour in $V(F)\setminus\{v\}$ (otherwise $v$ would form
its own component with no anchor); $u$ is adjacent to the image of
that neighbour, so there are at most $\Delta(G)$ choices.
\end{itemize}
In either case $c(F;G) \leq c(F';G) \cdot \Delta(G) \in
O(\Delta(G)^{|F|-|\sigma|})$. The label-extension condition is
immediate: an extension preserves the labelled-or-black-anchor
property.

($\Rightarrow$) Suppose $F$ has a component $C$ with no labelled or
black vertex. Take any $G_0 \in \Gcl$ and let $G_k$ be the disjoint
union of $G_0$ with $k$ disjoint copies of $C$. We realise each copy
by adding the appropriate independent red vertices; $G_k \in \Gcl$
(the construction preserves regularity, triangle-freeness, and the
black-set-size constraint). Each disjoint copy contributes one new embedding of $C$
and hence at least one new embedding of $F$, so $c(F;G_k) =
\Omega(k)$ while $\Delta(G_k) = \Delta(G_0)$. Hence $\rho(F;G_k)$ is
unbounded and $F$ is not a local $\sigma$-flag.
\end{proof}

\subsection{A size-5 objective}

Let $\sigma_O := \brrb$ be the type underlying the $4$-vertex
black-red-red-black 4-path. Its fully labelled version $\brrbmarked$ is
the size-$4$ $\sigma_O$-flag, the unit of $\Lcl^{\sigma_O}$.
By Lemma~\ref{lem:pentagon-local} $\sigma_O$ is a local type;
moreover $|\mathrm{Stab}(\theta)| = 2$ (the unique non-trivial
automorphism flips the 4-path end-to-end), so
$\llbracket\brrbmarked\rrbracket = (2/4!)\brrb = (1/12)\brrb$. By
Corollary~\ref{cor:unlabel-ext},
\[
\tfrac{1}{12}\,\phi(\brrb)
= \phi(\llbracket\brrbmarked\rrbracket)
= \phi(\llbracket\brrbmarked \cdot \ext_1^{\sigma_O}\rrbracket)
= \phi(\llbracket\ext_1^{\sigma_O}\rrbracket),
\]
where the last two equalities use the extension-multiplication identity
$\phi(\llbracket f\cdot\ext\rrbracket) = \phi(\llbracket f\rrbracket)$ and
that $\brrbmarked$ is the unit of $\Lcl^{\sigma_O}$.
Define the size-$5$ \emph{objective vector}
\[
O := \llbracket \ext_1^{\sigma_O} \rrbracket \in \Lcl^\emptyset_5.
\]
Bounding $\phi(O)$ above bounds
$\phi(\brrb) = 12\,\phi(O)$ and hence, via
Lemma~\ref{lem:brrb-suffices}, $P(H,v)/\Delta(H)^4$.

\subsection{A size-5 SDP and the 1/4 bound}

Take the basis $\Bcl = (F_1, \dots, F_\ell)$ of all local
$\emptyset$-flags of size $5$ in $\Gcl$, where $\ell = 58$. We denote
them by $F_j$ for $j \in \{1,\dots,58\}$ (a representative
selection drawn inline below as $\pentF{4}, \pentF{5}, \ldots$).

\begin{lemma}
\label{lem:size5-bound}
$\frac{1}{4}\emptyset - O \in \SemCone^\emptyset$.
\end{lemma}

\begin{proof}
We exhibit an explicit cone decomposition.

For each local type $\sigma$ of size $4$ and $i, j \in [4]$,
$\llbracket\ext_i^\sigma - \ext_j^\sigma\rrbracket \in
\SemCone^\emptyset$ by Corollary~\ref{cor:unlabel-ext}. We use five
specific instances, scaled by $120 = 5!$ to clear denominators.
Expanding the definition of $\ext_i^\sigma$ as a sum of size-$5$
$\sigma$-flags and applying $\llbracket\,\cdot\,\rrbracket$ gives the
right-hand side of each:
\begin{align}
120\,\llbracket\ext_2^{\sigma_1} - \ext_1^{\sigma_1}\rrbracket
&= -4\pentF{4} - 6\pentF{5} + 2\pentF{12} + 2\pentF{24} + 4\pentF{32},
\label{eq:sig1}\\
120\,\llbracket\ext_1^{\sigma_2} - \ext_2^{\sigma_2}\rrbracket
&= 24\pentF{6} - 4\pentF{13} - 4\pentF{15} - 2\pentF{18} - 2\pentF{25}
- 12\pentF{33} - 6\pentF{34},
\label{eq:sig2}\\
120\,\llbracket\ext_2^{\sigma_3} - \ext_1^{\sigma_3}\rrbracket
&= 2\pentF{11} + \pentF{12} - 2\pentF{15} + \pentF{22} + 2\pentF{24}
- 2\pentF{25},
\label{eq:sig3}\\
120\,\llbracket\ext_2^{\sigma_4} - \ext_1^{\sigma_4}\rrbracket
&= \pentF{8} + \pentF{12} - 2\pentF{15} + 4\pentF{31} + 4\pentF{32}
- 6\pentF{34},
\label{eq:sig4}\\
120\,\llbracket\ext_3^{\sigma_5} - \ext_1^{\sigma_5}\rrbracket
&= -2\pentF{16} - \pentF{17} + \pentF{37} + 2\pentF{55}.
\label{eq:sig5}
\end{align}
The types $\sigma_1, \dots, \sigma_5$ are the five $4$-vertex local
types $\pentsig{1}, \pentsig{2}, \pentsig{3}, \pentsig{4}, \pentsig{5}$.

Since the black set has size $\Delta(G)$,
$\rho(\vertex; G) = 1$ where $\vertex$ denotes the size-$1$
black-vertex $\emptyset$-flag. By
Corollary~\ref{cor:unlabel-ext} and direct computation
$\llbracket(\ext_1^{\vertex})^{4}\rrbracket = \frac{1}{5}\pentF{6}$,
so $\phi(\pentF{6}) = 5$ and
\begin{equation}
\label{eq:b1}
5\emptyset - \pentF{6} \in \SemCone^\emptyset.
\end{equation}

Next we exhibit two Cauchy--Schwarz blocks. Let
$\sigma_6 := \pentsig{6}$ be the local $3$-vertex type with skeleton
$\{0,1\}, \{0,2\}$ and colours $[R, B, B]$. Define
\[
f := -\pentcf{0}{3} + \tfrac{1}{4}\pentcf{0}{4} + \tfrac{1}{4}\pentcf{0}{5}
+ \tfrac{1}{2}\pentcf{0}{6},
\qquad
g := -\tfrac{1}{4}\pentcf{0}{4} + \tfrac{1}{4}\pentcf{0}{5},
\]
both in $\Lcl^{\sigma_6}$, where the flag $\pentcf{0}{3}$ uses the
$\sigma_6$-extended vertex $v_3$ coloured $R$. By
Lemma~\ref{lem:positivity}, $\llbracket f^2\rrbracket, \llbracket
g^2\rrbracket \in \SemCone^\emptyset$. Direct chain-rule
expansion of each product followed by averaging gives
\begin{align*}
120\,\llbracket f^2\rrbracket
&= 4\pentF{4} - \pentF{8} - 2\pentF{11} + \pentF{13}
+ \tfrac{1}{4}\pentF{18} - \pentF{22} - 4\pentF{31} + 3\pentF{33}
+ \tfrac{1}{4}\pentF{40} + \tfrac{1}{4}\pentF{55},\\
120\,\llbracket g^2\rrbracket
&= \tfrac{1}{4}\pentF{18} - \tfrac{1}{4}\pentF{40} - \tfrac{1}{4}\pentF{55},
\end{align*}
and summing,
\begin{equation}
\label{eq:cs0}
4\pentF{4} - \pentF{8} - 2\pentF{11} + \pentF{13}
+ \tfrac{1}{2}\pentF{18} - \pentF{22} - 4\pentF{31} + 3\pentF{33}
\in \SemCone^\emptyset.
\end{equation}

Let $\sigma_7 := \pentsig{7}$ be the local $3$-vertex type with the
same skeleton as $\sigma_6$ and colours $[R, B, R]$. Define
\[
h := \tfrac{1}{2}\pentcf{1}{4} - \pentcf{1}{5}
+ \tfrac{1}{2}\pentcf{1}{7},
\qquad
g' := \pentcf{1}{3} - \pentcf{1}{4} + \pentcf{1}{5} - \pentcf{1}{7},
\]
both in $\Lcl^{\sigma_7}$. By Lemma~\ref{lem:positivity},
$\llbracket h^2\rrbracket, \llbracket g'^2\rrbracket \in
\SemCone^\emptyset$. Chain-rule expansion and averaging give
\begin{align*}
120\,\llbracket h^2 \rrbracket
&= -2\pentF{9} + \pentF{15} + 2\pentF{16}
+ \tfrac{1}{2}\pentF{25} + \tfrac{3}{2}\pentF{34}
- \pentF{37} - 2\pentF{55},\\
120\,\llbracket g'^2 \rrbracket
&= 6\pentF{5} - 4\pentF{12} + 4\pentF{15} + 2\pentF{16} + 2\pentF{17}
- 4\pentF{24} + 2\pentF{25} - 8\pentF{32}\\
&\quad + 6\pentF{34} - 2\pentF{37} - 4\pentF{55},
\end{align*}
and summing,
\begin{equation}
\label{eq:cs1}
6\pentF{5} - 2\pentF{9} - 4\pentF{12} + 5\pentF{15} + 4\pentF{16}
+ 2\pentF{17} - 4\pentF{24} + \tfrac{5}{2}\pentF{25}
- 8\pentF{32} + \tfrac{15}{2}\pentF{34} - 3\pentF{37} - 6\pentF{55}
\in \SemCone^\emptyset.
\end{equation}

The convex combination
\[
1\cdot\eqref{eq:sig1}
+ \tfrac{1}{4}\cdot\eqref{eq:sig2}
+ 1\cdot\eqref{eq:sig3}
+ 1\cdot\eqref{eq:sig4}
+ 2\cdot\eqref{eq:sig5}
+ 6\cdot\eqref{eq:b1}
\]
yields the linear sum
\begin{multline}
\label{eq:linsum}
30\emptyset - 4\pentF{4} - 6\pentF{5} + \pentF{8} + 2\pentF{11}
+ 4\pentF{12} - \pentF{13} - 5\pentF{15} - 4\pentF{16} - 2\pentF{17}\\
- \tfrac{1}{2}\pentF{18} + \pentF{22} + 4\pentF{24}
- \tfrac{5}{2}\pentF{25} + 4\pentF{31} + 8\pentF{32} - 3\pentF{33}
- \tfrac{15}{2}\pentF{34} + 2\pentF{37} + 4\pentF{55}
\in \SemCone^\emptyset.
\end{multline}
Adding~\eqref{eq:cs0}, \eqref{eq:cs1}, \eqref{eq:linsum} yields
\[
30\emptyset - 2\pentF{9} - \pentF{37} - 2\pentF{55} \in \SemCone^\emptyset.
\]
Computing the right-hand side of $O = \llbracket\ext_1^{\sigma_O}
\rrbracket$ by expanding $\ext_1^{\sigma_O}$ as a sum of size-$5$
$\sigma_O$-flags (the three flags differ by the adjacencies of the
single new vertex, the red neighbour of position $1$, to the remaining
labelled positions) and applying $\llbracket\,\cdot\,\rrbracket$ gives
\[
120\,O = 2\pentF{9} + \pentF{37} + 2\pentF{55}.
\]
Therefore $30\emptyset - 120\,O \in \SemCone^\emptyset$, equivalently
$\tfrac{1}{4}\emptyset - O \in \SemCone^\emptyset$.
\end{proof}

\begin{proof}[Proof of Theorem~\ref{thm:simple}]
By Lemma~\ref{lem:size5-bound}, $\phi(O) \leq 1/4$ for every $\phi \in
\Phi^\emptyset$, so $\phi(\brrb) = 12\phi(O) \leq 3$. For any
$\Delta$-increasing sequence $(G_k) \subseteq \Gcl$ with limit
functional $\phi$,
\[
\lim_k \frac{c(\brrb; G_k)}{\binom{\Delta(G_k)}{4}} \leq 3,
\]
which with $\binom{\Delta}{4} = \Delta^4/4! + o(\Delta^4)$ gives
$\lim c(\brrb;G_k)/\Delta(G_k)^4 \leq 3/24 = 1/8$. By
Lemma~\ref{lem:brrb-suffices}, $P(H,v)/\Delta(H)^4 \lesssim 1/8$ over
regular triangle-free $H$. By Lemma~\ref{lem:local-count},
$P(H)/(|H|\Delta(H)^4) \lesssim 1/40$. Finally
Lemmas~\ref{lem:asymp-suffices} and~\ref{lem:regular-suffices} remove
the regularity and asymptotic qualifiers.
\end{proof}

\subsection{Tightness at the per-vertex level}

Tightening only the per-vertex bound $P(G,v)/\Delta(G)^4$ cannot
improve the constant $1/40$ of Theorem~\ref{thm:simple}, since a
simple construction realises the size-$5$ certificate's $1/8$ bound at
the vertex level.

\begin{lemma}
\label{lem:per-vertex-tight}
For every even $k$ there exists an (almost-)regular triangle-free
graph $G_k$ with $\Delta(G_k) = k+1$ and a vertex $v \in V(G_k)$ on
exactly $k^4/8$ pentagons, so $P(G_k,v)/\Delta(G_k)^4 \to 1/8$.
\end{lemma}

\begin{proof}
Let $G_k$ be the $(k/2)$-blowup of $C_6$ (six supernodes of size $k/2$
each, with each pair of consecutive supernodes inducing a complete
bipartite graph $K_{k/2,k/2}$) together with one extra vertex $v$
joined to all $k$ vertices of two antipodal supernodes
(Figure~\ref{fig:hexagon}). A triangle in $G_k$ would project to a
triangle in $C_6$ or use $v$; the first is impossible because $C_6$ is
triangle-free, the second because the two neighbouring supernodes of
$v$ are not adjacent in the hexagon. The maximum degree is $k$ except
at the $k$ vertices in the two supernodes adjacent to $v$, where it is
$k+1$.

Each pentagon through $v$ uses $v$, one vertex from each supernode
adjacent to $v$ (a factor of $(k/2)^2$), and two further vertices
along the hexagon in either the clockwise or counterclockwise
direction (a factor of $2 \cdot (k/2)^2$). The product is
$(k/2)^2 \cdot 2 \cdot (k/2)^2 = k^4/8$.
\end{proof}

\begin{figure}[ht]
\centering
\includegraphics[scale=2.5]{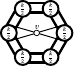}
\caption{The hexagon-blowup-plus-extra-vertex construction of
Lemma~\ref{lem:per-vertex-tight}.}
\label{fig:hexagon}
\end{figure}

This obstruction motivates the richer functional $Q$ of the next
section, which combines $P(G,v)$ with neighbour contributions and
escapes the per-vertex ceiling.

\section{A tighter pentagon bound}
\label{sec:tight}

We prove Theorem~\ref{thm:tight}. The architecture follows
Section~\ref{sec:simple}: reduce to a per-vertex asymptotic bound,
express as $\phi$ of a flag-algebra objective, and bound by a
certificate. The objective is no longer a single 4-path count; instead it
is a per-vertex functional that combines $P(G,v)$ with a weighted sum
over neighbours.

\subsection{\texorpdfstring{The functional $Q(G,v)$}{The functional Q(G,v)}}

For a regular triangle-free graph $G$ and $v \in V(G)$, set
\[
Q(G,v) := \Delta(G)\cdot P(G,v) + \sum_{u \in N(v)} P(G,u).
\]
Summing over $v$ and using regularity, the double sum reorders as
\[
\sum_v \sum_{u \in N(v)} P(G,u)
= \sum_u \deg(u)\,P(G,u)
= \Delta(G)\sum_u P(G,u),
\]
which gives
\begin{align*}
\sum_{v \in V(G)} Q(G,v)
&= \Delta(G)\sum_v P(G,v) + \Delta(G)\sum_u P(G,u)\\
&= 2\Delta(G)\sum_v P(G,v) = 10\Delta(G) P(G),
\end{align*}
the last step using $\sum_v P(G,v) = 5P(G)$, as each pentagon is
counted at its five vertices.
A bound $Q(G,v) \lesssim c \Delta(G)^5$ therefore yields
$P(G) \lesssim (c/10)\,|G|\Delta(G)^4$. Taking $c = 0.2073$ delivers
Theorem~\ref{thm:tight}.

\begin{lemma}
\label{lem:Q-v}
For every regular triangle-free $G$ and $v \in V(G)$,
$Q(G,v) \lesssim 0.2073\,\Delta(G)^5$ as $\Delta(G) \to \infty$.
\end{lemma}

The rest of this section proves Lemma~\ref{lem:Q-v}. The setup uses
the same 2-coloured class $\Gcl$ as Section~\ref{sec:simple}; the
contribution $P(G,v)$ corresponds, as in
Lemma~\ref{lem:brrb-suffices}, to black-red-red-black 4-paths in the
auxiliary graph. Each pentagon $u u_1 u_2 u_3 u_4 u$ with $u \in N(v)$
visits the black set $N(v)$ at most twice, since the black set is
independent; in the 2-coloured class each such pentagon corresponds
either to the size-$5$ $\emptyset$-flag $\pentF{55}$ (two black
vertices among the five, necessarily non-adjacent since the black set
is independent) or to $\pentF{56}$ (one black, four red). Direct enumeration
gives
\[
\rho(\pentF{56} + 2\pentF{55}; G')
= \frac{\sum_{u \in N(v)} P(G,u)}{\binom{\Delta(G)}{5}} + o(1),
\]
where $G'$ is the 2-colouring of $G$ as in
Lemma~\ref{lem:brrb-suffices}; the coefficient $2$ on $\pentF{55}$
records that a pentagon with two black vertices is counted at each of
them in $\sum_{u \in N(v)} P(G,u)$.

\subsection{A size-8 objective}

Using Corollary~\ref{cor:unlabel-ext} to lift each contribution to
size $n$, with $\sigma_1 := \brrb$ and $\sigma_2, \sigma_3$ the
$5$-vertex types obtained by fully labelling $\pentF{56}$ and
$\pentF{55}$ respectively, the extension $\ext_1$ acting at the
(labelled) black endpoint in each case, define
\[
O_Q := \bigl\llbracket \brrbmarked \cdot (\ext_1^{\sigma_1})^{n-4}\bigr\rrbracket
+ \bigl\llbracket F_{56}^\bullet \cdot (\ext_1^{\sigma_2})^{n-5}\bigr\rrbracket
+ 2\bigl\llbracket F_{55}^\bullet \cdot (\ext_1^{\sigma_3})^{n-5}\bigr\rrbracket,
\]
where $F_{55}^\bullet, F_{56}^\bullet$ are the labelled versions of
$\pentF{55}, \pentF{56}$. We take $n = 8$.

The size-$8$ unlabelled flag basis $\Bcl_8 = (F_1, \dots, F_m)$ of
$\Gloc^\emptyset$ in the class $\Gcl$ has $m = 9295$ elements.

\begin{lemma}
\label{lem:O-Q-bound}
$0.4146\emptyset - O_Q \in \SemCone^\emptyset$. In particular, for
every $\phi \in \Phi^\emptyset$, $\phi(O_Q) \leq 0.4146$.
\end{lemma}

We prove Lemma~\ref{lem:O-Q-bound} by verifying an explicit size-$8$
SDP certificate. Positivity of each Cauchy--Schwarz block, together
with the linear-residual cone identity, establishes the dual
feasibility statement $0.4146\emptyset - O_Q \in \SemCone^\emptyset$;
applying any $\phi \in \Phi^\emptyset$ then yields
$\phi(O_Q) \leq 0.4146$. Section~\ref{sec:certificates} carries out this
verification via the size-$8$ certificate.

\begin{proof}[Proof of Lemma~\ref{lem:Q-v}]
By Lemma~\ref{lem:O-Q-bound} and the chain-rule expansion of $O_Q$
into per-flag densities (the coefficients $\tfrac{2}{4!}, \tfrac{2}{5!}$
arising from unlabelling the size-$4$ and size-$5$ flags;
Section~\ref{sec:certificates}),
\[
\tfrac{2}{4!}\phi(\brrb) + \tfrac{2}{5!}\phi(\pentF{56} + 2\pentF{55})
\leq 0.4146.
\]
Since $\phi(F) \sim |F|!\,c(F;G')/\Delta(G)^{|F|}$ by the asymptotic
identity $\binom{\Delta}{k} \sim \Delta^k/k!$ of
Lemma~\ref{lem:asymp-averaging}, the factor $1/|F|!$ in each coefficient
cancels; dividing by the factor $2$ common to the coefficients --- the
normalisation factor in $\phi(O_Q) = 2\,Q(G,v)/\Delta(G)^5$
(Section~\ref{sec:certificates}), not the coefficient $2$ on
$\pentF{55}$ --- halves the bound to $0.2073$:
\[
\frac{c(\brrb;G')}{\Delta(G)^4}
+ \frac{c(\pentF{56};G')}{\Delta(G)^5}
+ 2\,\frac{c(\pentF{55};G')}{\Delta(G)^5}
\leq 0.2073 + o(1).
\]
Multiplying through by $\Delta(G)^5$ and substituting the
correspondences
$\Delta(G)\,c(\brrb;G') = \Delta(G)\,P(G,v) + O(\Delta(G)^4)$
and
$c(\pentF{56};G') + 2c(\pentF{55};G') = \sum_{u \in N(v)} P(G,u)
+ O(\Delta(G)^4)$
gives $Q(G,v) \leq 0.2073\,\Delta(G)^5 + o(\Delta(G)^5)$, which is the
required asymptotic.
\end{proof}

\begin{proof}[Proof of Theorem~\ref{thm:tight}]
Sum Lemma~\ref{lem:Q-v} over $v \in V(G)$ and use $\sum_v Q(G,v) =
10\Delta(G) P(G)$ to get $P(G) \lesssim 0.02073\,|G|\Delta(G)^4$ over
regular triangle-free $G$. Reduce to regular via
Lemma~\ref{lem:regular-suffices}, and remove the asymptotic qualifier
via Lemma~\ref{lem:asymp-suffices}.
\end{proof}

\section{The size-5 and size-8 certificates}
\label{sec:certificates}

This section describes the two SDP certificates used in
Sections~\ref{sec:simple} and~\ref{sec:tight} as mathematical objects,
and proves the two ingredients of Lemma~\ref{lem:O-Q-bound} deferred
from Section~\ref{sec:tight}. The certificate data
itself --- the flag bases, block matrices, and rationalised witnesses
--- together with the generator that produces it and the Lean~4
formalisation that verifies it, is available
at~\cite{daveyLocalFlags2024Repo}.

Both certificates are finite collections of rational data admitting a
finite-arithmetic verification of the cone decomposition that
Lemmas~\ref{lem:size5-bound} and~\ref{lem:O-Q-bound} imply. The witnesses
grow with the basis --- the size-$8$ one is exponentially larger than the
size-$5$ --- but the structural complexity of the verification, a finite
integer identity, does not.

\subsection{The size-5 certificate}

The size-$5$ certificate is the explicit cone decomposition that the
proof of Lemma~\ref{lem:size5-bound} exhibits. It uses the size-$5$
local-$\emptyset$-flag basis of cardinality $\ell = 58$, the five
extension-difference vectors~\eqref{eq:sig1}--\eqref{eq:sig5}, the
black-vertex normalisation~\eqref{eq:b1}, and the two Cauchy--Schwarz
blocks~\eqref{eq:cs0} and~\eqref{eq:cs1}. The arithmetic
identity~\eqref{eq:linsum} plus the addition step assembling all three
into $30\emptyset - 120\,O$ is a finite rational arithmetic check on
$58$-vector coefficients.

\subsection{The size-8 certificate: structure}

The size-$8$ certificate is supported on a basis of $m = 9295$
unlabelled local $\emptyset$-flags of size $8$ in the 2-coloured class
$\Gcl$, taken up to isomorphism, and decomposes the bound into $278$
Cauchy--Schwarz blocks, one per local type $\sigma$ --- the block for
$\sigma$ is built from $\sigma$-flags of size $(8 + |\sigma|)/2$, that
is, adding $(8 - |\sigma|)/2$ unlabelled vertices to $\sigma$ ---
together with a linear residual. The blocks
vary in $\sigma$-type and inner dimension; a common integer scale
$\linscale := 10^{12}$ carries the rationalisation of the floating-point
dual matrix $Y_\mathrm{float}$: its $12$-digit rounding
$Y_\mathrm{int} = \lceil \linscale \cdot Y_\mathrm{float} \rfloor$ gives
the rational matrix $Y_\mathrm{rat} = Y_\mathrm{int}/\linscale$. We record the explicit per-block dimensions and
magnitudes with the formalisation~\cite{daveyLocalFlags2024Repo};
Appendix~\ref{app:cert-data} traces one block in full.

Each block $i$ contributes a cone element
\begin{equation}
\label{eq:csblock}
\mathrm{csBlock}_i
:= \biggl\llbracket\sum_{k=0}^{d(i)-1} \tfrac{D_k^{(i)}}{s_i}
\;\mathrm{col}_k^{(i)} \cdot \mathrm{col}_k^{(i)}\biggr\rrbracket
\in \SemCone^\emptyset,
\end{equation}
where $\mathrm{col}_k^{(i)} := \sum_{p=0}^{d(i)-1} L_{p,k}^{(i)}\,
\mathrm{basis}^{(i)}_p$ is the $k$-th column of the certificate's integer
$L$-matrix, $D_k^{(i)}$ is the $k$-th LDL pivot, and $s_i > 0$ is a
shared block-level integer denominator that absorbs the column-wise
scale gaps of $L_i$.

The decomposition that proves Lemma~\ref{lem:O-Q-bound} is
\begin{equation}
\label{eq:size8-decomp}
0.4146\emptyset - O_Q
= \mathrm{linSum}_Q + \sum_{i=0}^{277} \mathrm{csBlock}_i,
\end{equation}
where $\mathrm{linSum}_Q$ is the linear residual built from
extension-difference cone elements and the black-vertex normalisation
relation $\rho(\vertex; G) = 1$, all lifted from size $5$ to size $8$
via Corollary~\ref{cor:unlabel-ext}.

The decomposition~\eqref{eq:size8-decomp} requires three ingredients,
which the subsections below prove.
\begin{enumerate}[label=(C\arabic*)]
\item \label{cert:per-block} Each $\mathrm{csBlock}_i$ lies in
$\SemCone^\emptyset$ (per-block positivity).
\item \label{cert:linres} $\mathrm{linSum}_Q$ lies in
$\SemCone^\emptyset$ (linear-residual positivity).
\item \label{cert:identity} The
identity~\eqref{eq:size8-decomp} holds in $\Lcl^\emptyset_8$
(arithmetic identity).
\end{enumerate}

\subsection{Per-block positivity \ref{cert:per-block}}
\label{subsec:cert-per-block}

For each block $i$ the certificate exhibits an integer LDL identity
\begin{equation}
\label{eq:ldl}
L_i \cdot \operatorname{diag}(D_i) \cdot L_i^\top
= s_i \cdot \bigl(Y_i + \lambda_i \cdot I_{d(i)}\bigr),
\end{equation}
where:
\begin{itemize}
\item $L_i \in \mathbb{Z}^{d(i) \times d(i)}$ is a lower-triangular
matrix with positive integer diagonal entries;
\item $D_i \in \mathbb{Z}_{>0}^{d(i)}$ is the vector of LDL pivots,
strictly positive;
\item $Y_i \in \mathbb{Z}^{d(i) \times d(i)}$ is the integer matrix
$\linscale \cdot Y_{i,\mathrm{rat}}$, which we read off the dual
block matrix at $12$-digit rationalisation;
\item $\lambda_i \in \mathbb{Q}_{>0}$ is a per-block Tikhonov shift
absorbing the solver's dual-feasibility residual;
\item $s_i \in \mathbb{Z}_{>0}$ is a shared block-level integer
denominator absorbing the column-wise scale gaps in $L_i$.
\end{itemize}
The identity~\eqref{eq:ldl} is a finite integer matrix equality. One
verifies it entry-wise: for each pair $(p, q)$ with $p \geq q$ in
$\{0, \dots, d(i)-1\}$,
\[
\sum_{k=0}^{q} L_{p,k}^{(i)}\,D_k^{(i)}\,L_{q,k}^{(i)}
= s_i \cdot \bigl(Y_i[p,q] + \lambda_i\cdot[p=q]\bigr).
\]
Since $D_i > 0$ entry-wise and $L_i$ is lower-triangular,
\[
Y_i + \lambda_i I_{d(i)}
= \tfrac{1}{s_i}\,L_i\operatorname{diag}(D_i) L_i^\top
\succeq 0,
\]
so $Y_i + \lambda_i I$ is positive semidefinite at integer scale. The
identity~\eqref{eq:ldl} together with $D_i > 0$ is therefore an
explicit PSD certificate for the rationalised dual block.

Each column $\mathrm{col}_k^{(i)} \in \Lcl^{\sigma(i)}$ lies
in the $\sigma(i)$-flag algebra at the inner basis size; squaring and
averaging gives, by Lemma~\ref{lem:positivity},
$\llbracket(\mathrm{col}_k^{(i)})^2\rrbracket \in \SemCone^\emptyset$.
Multiplying by the non-negative rational $D_k^{(i)}/s_i > 0$ and
summing over $k$ produces $\mathrm{csBlock}_i \in \SemCone^\emptyset$,
which is~\ref{cert:per-block}.

\subsection{Linear-residual positivity \ref{cert:linres}}

By construction, the linear residual $\mathrm{linSum}_Q$ is a
non-negative integer combination of three families of cone-positive
elements:
\begin{enumerate}[label=(R\arabic*)]
\item \emph{Extension-difference vectors.} For each local type
$\sigma$ entering the certificate with $|\sigma| \in \{4, 5, 6, 7\}$ and
each pair $i, j \in [|\sigma|]$, the vector
$\llbracket\ext_i^\sigma - \ext_j^\sigma\rrbracket \in
\SemCone^\emptyset$ by Corollary~\ref{cor:unlabel-ext}. The certificate uses
a finite set $\mathcal{E}$ of such pairs.
\item \emph{Black-vertex normalisation lifted to size $8$.} For
$k \in [4]$, the relation $\phi(\llbracket(\ext_1^{B_k})^{8-k}\rrbracket)
= 1$ in $\Lcl^\emptyset_8$, where $B_k$ is the $k$-vertex empty
flag with $k$ black vertices, gives cone elements
$m_k \emptyset - \mathrm{lift}_k \in \SemCone^\emptyset$ for
explicit constants $m_k$ encoding the size-$8$ lift of $B_k$.
\item \emph{Black-set cardinality.} The relation $|B(G)| = \Delta(G)$
gives $\rho(B_1;G) = 1$ for every $G \in \Gcl$, so
$\emptyset - B_1 \in \SemCone^\emptyset$.
\end{enumerate}
The certificate writes $\mathrm{linSum}_Q$ as
\[
\mathrm{linSum}_Q = \sum_{e \in \mathcal{E}} \alpha_e
\llbracket\ext_{i(e)}^{\sigma(e)} - \ext_{j(e)}^{\sigma(e)}\rrbracket
+ \sum_{k=1}^4 \beta_k(m_k \emptyset - \mathrm{lift}_k)
+ \gamma(\emptyset - B_1),
\]
with $\alpha_e, \beta_k, \gamma \in \mathbb{Q}_{\geq 0}$ explicit. Each
summand lies in $\SemCone^\emptyset$ by Corollary~\ref{cor:unlabel-ext}
and the cited identities; non-negative linear combination preserves
positivity.

\subsection{The arithmetic identity \ref{cert:identity}}
\label{subsec:cert-identity}

The identity~\eqref{eq:size8-decomp} is an equality of elements of
$\Lcl^\emptyset_8 = \R^{9295}$. We prove it by expanding each side in
the basis $\Bcl_8 = (F_1, \dots, F_{9295})$ and checking entry-wise.

Write the size-$8$ basis-density coefficients of $O_Q$ as
$O_Q = \sum_{j=1}^{9295} \mathrm{coef}_j\,F_j$, with
$\mathrm{coef}_j = -\mathrm{target}_j / \linscale$, where
$\mathrm{target}_j \in \mathbb{Z}$ is the integer target that the
certificate carries. The factor $-1/\linscale$ rescales targets from
the integer certificate units (in which the cone identity reads as an integer
equality) back to the eval-level rational unit of
$O_Q \in \Lcl^\emptyset_8$. The sign flip places the cone inequality
in the standard form $\lambda\emptyset - O_Q \in \SemCone^\emptyset$,
as in Section~\ref{sec:sdp-method}.

Write each csBlock as $\mathrm{csBlock}_i = \sum_{j=1}^{9295}
\mathrm{cs}_{i,j}\,F_j$, where chain-rule expansion
(Definition~\ref{def:product}) of each $\mathrm{col}_k^{(i)} \cdot
\mathrm{col}_k^{(i)}$ followed by averaging
$\llbracket\,\cdot\,\rrbracket$ and basis enumeration produces the
coefficient $\mathrm{cs}_{i,j}$. Write
$\mathrm{linSum}_Q = \sum_{j=1}^{9295} \ell_j F_j$ analogously.

The arithmetic identity~\eqref{eq:size8-decomp} unfolds as the system
of $9295$ equations
\begin{equation}
\label{eq:per-flag}
0.4146 \cdot [F_j = \emptyset] - \mathrm{coef}_j
= \ell_j + \sum_{i=0}^{277} \mathrm{cs}_{i,j},
\qquad j = 1, \dots, 9295,
\end{equation}
where $[F_j = \emptyset]$ is $1$ if $F_j$ is the (size-$8$)
empty-flag basis representative and $0$ otherwise. Clearing the common
denominator $\linscale^2 = 10^{24}$ and multiplying both sides by
$\linscale^2$ converts~\eqref{eq:per-flag} into a system of $9295$
integer identities
\begin{equation}
\label{eq:per-flag-int}
\linscale^2 \cdot \bigl(0.4146 \cdot [F_j = \emptyset]\bigr)
+ \linscale \cdot \mathrm{target}_j
= \linscale^2 \cdot \ell_j
+ \sum_{i=0}^{277} \linscale^2 \cdot \mathrm{cs}_{i,j},
\end{equation}
each side an integer. The per-block chain-rule expansion of
Section~\ref{subsec:cert-per-block} and the linear-residual
expansion of the previous subsection reduce
verification of~\eqref{eq:per-flag-int} to a finite sum of finite
integer products organised by basis index $j$. We discuss the slack
that arises in this verification next.

\subsection{Slack budget}
\label{subsec:cert-slack}

The integer identity~\eqref{eq:per-flag-int} is exact in the certificate's
PSD ingredients $L_i, D_i, \lambda_i, s_i$, but the rationalised dual
matrix $Y_\mathrm{rat}$ inherits a small residual from
rationalising the solver's floating-point output. The solver's primal-dual
optimum is $\approx 0.41458$ with reported max dual-feasibility
residual $\max_k |\trace(M_k \cdot Y_\mathrm{float}) - c_k| \approx
10^{-12}$. After $12$-digit rationalisation, the per-flag integer
residual is
\[
\mathrm{residual}_j :=
s_\mathrm{tot} \cdot \Bigl(\trace(M_j \cdot Y_\mathrm{rat})
+ \sum_i \lambda_i\, \trace(M_j|_i) - c_j\Bigr),
\]
where $s_\mathrm{tot} = \linscale^2 = 10^{24}$ is the shared integer
denominator, $M_j|_i$ is the restriction of $M_j$ to block $i$ with
per-block Tikhonov shift $\lambda_i$, and $M_j$ is the structure matrix
of the $j$-th of the
$9295$ dual constraints (assembled from the block structure matrices
$M_\sigma^{(i)}$). The
aggregated weighted slack
\[
\totalSlack
:= \Bigl|\sum_{j=1}^{9295} x_j \cdot \mathrm{residual}_j\Bigr|
\]
quantifies the error that the identity~\eqref{eq:per-flag-int} carries.

We exhibit a finite slack budget that absorbs this residual at the
tight constant $2073/10000$ --- the $O_Q$-bound $0.4146$ expressed at
the $Q$ normalisation, $2073/10000 = 0.4146/2$, the factor $2$ being
the asymptotic identity $2\,Q(G_k,v_k)/\Delta(G_k)^5 = \phi(O_Q) +
o(1)$ of Lemma~\ref{lem:basis-identity}.

\begin{lemma}[slack budget]
\label{lem:slack}
At the tight rational $2073/10000$ and scale $\linscale = 10^{12}$,
the aggregated weighted slack satisfies
\[
\totalSlack \leq \slackBudget,
\qquad
\slackBudget := 10^{19}.
\]
The measured value is $\totalSlack \approx 5.48 \times 10^{18}$
($= 5.48 \times 10^{-6} \cdot \linscale^2$), so the budget is met with
a safety ratio of $\approx 1.82\times$ in $L$-space, equivalently
$\approx 3.65\times$ in $O_Q$-space.
\end{lemma}

\begin{proof}
Direct integer arithmetic. The $9295$ values
$\mathrm{residual}_j$ are integers read off the certificate; the dual
weights $x_j$ are integers from the certificate's rationalisation. The signed
weighted sum $\sum_j x_j \cdot \mathrm{residual}_j$ is an integer that
a single pass over the certificate data computes. Its absolute
value is at most $\approx 5.48 \times 10^{18}$ by direct integer
addition; $10^{19}$ bounds it from above.
\end{proof}

Together, the per-block PSD identities of
Section~\ref{subsec:cert-per-block}, the linear-residual identity, the
arithmetic identity~\eqref{eq:per-flag-int}, and the slack-budget
bound of Lemma~\ref{lem:slack} give the proof of
Lemma~\ref{lem:O-Q-bound}.

\begin{proof}[Proof of Lemma~\ref{lem:O-Q-bound}]
By Lemma~\ref{lem:positivity} and~\eqref{eq:ldl} per block, each
$\mathrm{csBlock}_i \in \SemCone^\emptyset$. The linear residual
$\mathrm{linSum}_Q \in \SemCone^\emptyset$ as a non-negative
combination of extension differences and the black-vertex
normalisation. The arithmetic identity~\eqref{eq:per-flag-int} holds
up to an aggregated weighted slack at most $\slackBudget = 10^{19}$ in
integer L-units, which Lemma~\ref{lem:slack} satisfies with a
safety margin of $\approx 1.82\times$. We choose the certificate's tight pair
$(2073, 10000)$ so that the budget exactly covers the
aggregated slack on the right-hand side; combining
gives~\eqref{eq:size8-decomp}, hence $0.4146\emptyset - O_Q \in
\SemCone^\emptyset$.
\end{proof}

\subsection{Basis combinatorial identity}

Lemma~\ref{lem:O-Q-bound} bounds the flag-algebra objective $O_Q$; to
turn it into a bound on the pentagon count we need the basis
combinatorial identity relating $2\,Q(G,v)/\Delta(G)^5$ to the
densities of the size-$8$ unlabelled flag basis $\Bcl_8$.

\begin{lemma}[basis combinatorial identity]
\label{lem:basis-identity}
For every triangle-free regular sequence $(G_k, v_k)$ with strictly
increasing $\Delta(G_k) \to \infty$,
\[
\frac{2 \cdot Q(G_k, v_k)}{\Delta(G_k)^5}
= \sum_{j=1}^{9295}
\mathrm{coef}_j \cdot \rho(\mathrm{basis}_j; G^c_k) + o(1),
\]
where $G^c_k$ is the canonical $2$-colouring of $G_k$ at the basepoint
$v_k$ and $\mathrm{coef}_j = -\mathrm{target}_j/\linscale$.
\end{lemma}

\begin{proof}
A combinatorial bijection identifies pentagon-extension tuples on
$(G_k, v_k)$ with labelled induced embeddings of the size-$8$ basis
flags into the canonical $2$-colouring $G^c_k$. Each pentagon
contributes either via the $\brrb$ component ($v \in V(C_5)$) or via
the $\pentF{56} + 2\pentF{55}$ component ($u \in N(v) \cap V(C_5)$),
and the factor of $2$ on the left-hand side absorbs the
double-counting that occurs in the $\sum_{u \in N(v)} P(G, u)$ term
(each pentagon containing two vertices of $N(v)$ contributes twice).

The per-class target is
\[
\mathrm{target}_j = -\Bigl\lceil m_j \cdot \linscale \big/ \tfrac{8!}{3!}
\Bigr\rfloor \;\leq\; 0,
\qquad \tfrac{8!}{3!} = 6720,
\]
where $m_j \in \mathbb{Z}_{\geq 0}$ is the objective's integer weight
for $\mathrm{basis}_j$ (a pentagon-extension count) and $\lceil\cdot\rfloor$
denotes rounding to the nearest integer at the $12$-digit scale. The
normalising denominator is the descending factorial
$8!/3! = 6720$ --- the injective count of the five pentagon vertices in
the size-$8$ frame --- not a per-flag automorphism count, which instead
sits in the density $\rho$. The non-positive sign is that of the
solver's ``$\min\,{-}(\cdot)$'' convention
(Section~\ref{sec:sdp-method}); it makes
$\mathrm{coef}_j = -\mathrm{target}_j/\linscale \approx m_j/6720 \geq 0$
(exact up to the $12$-digit rounding, whose residual the slack budget of
Section~\ref{subsec:cert-slack} absorbs), so that
$\phi(O_Q) = \sum_j \mathrm{coef}_j\,\rho(\mathrm{basis}_j;\cdot) \geq 0$
and the bound $\phi(O_Q) \leq 0.4146$ is a genuine upper bound on the
non-negative quantity $2\,Q(G_k,v_k)/\Delta(G_k)^5$.

Normalising induced counts to densities via
$\rho(\mathrm{basis}_j; G^c_k) = c(\mathrm{basis}_j; G^c_k) /
\binom{\Delta(G_k)}{8}$ and using $\binom{\Delta}{8} = \Delta^8/8! +
o(\Delta^8)$ absorbs the asymptotic factor; the $o(1)$ correction
comes from the $\binom{\Delta - O(1)}{8}/\binom{\Delta}{8} = 1 +
O(1/\Delta)$ ratios that enter the canonical-colouring restriction.
The detailed combinatorial enumeration follows the architecture of
Section~\ref{sec:simple} for the size-$5$ case
(Lemma~\ref{lem:size5-bound}) generalised to size $8$ via the
per-isomorphism-class count above.
\end{proof}

The combination of Lemma~\ref{lem:basis-identity} and
Lemma~\ref{lem:O-Q-bound} completes the proof of
Lemma~\ref{lem:Q-v} and hence of Theorem~\ref{thm:tight}.

\section{The Clebsch extremum}
\label{sec:clebsch}

We now prove Lemma~\ref{lem:clebsch}, establish the conjecture at
maximum degree five together with its extremal characterisation
(Theorem~\ref{thm:delta5}), and survey the empirical evidence in
support of the general conjecture.

\subsection{The Clebsch graph}

We may define the \emph{Clebsch graph} $\mathrm{Cl}$ as the folded
$5$-cube: take the vertex set $\{0,1\}^5$ identifying antipodes, so
$|V(\mathrm{Cl})| = 16$, and join two equivalence classes if they
contain representatives at Hamming distance $1$. Equivalently,
$\mathrm{Cl}$ is the unique strongly regular graph
$\mathrm{SRG}(16,5,0,2)$~\cite{brouwerDistanceRegularGraphs1989}: it is
$5$-regular on $16$ vertices, every edge lies in $0$ triangles, and
every non-edge has exactly $2$ common neighbours. In particular,
$\mathrm{Cl}$ is triangle-free with $\Delta(\mathrm{Cl}) = 5$. Its
spectrum is $\{5, 1^{10}, (-3)^5\}$.

\begin{lemma}
\label{lem:clebsch-count}
$\mathrm{Cl}$ contains exactly $192$ induced copies of $C_5$.
\end{lemma}

\begin{proof}
We count closed $5$-walks $W$ in $\mathrm{Cl}$ that visit five distinct
vertices and have no chord. A closed $5$-walk corresponds to a cyclic
sequence $(v_0, v_1, v_2, v_3, v_4)$ with $v_i v_{i+1} \in E$ (indices
mod $5$). The number of closed walks of length $5$ from a fixed vertex
equals $A^5_{vv}$ where $A$ is the adjacency matrix; summing,
$\sum_v A^5_{vv} = \trace(A^5)$. From the spectrum,
$\trace(A^5) = 5^5 + 10\cdot 1^5 + 5\cdot(-3)^5 = 3125 + 10 - 1215 = 1920$.

Each closed $5$-walk that is a pentagon (five distinct vertices, no
chord) contributes exactly $10$ times to $\trace(A^5)$ (a pentagon has
$10$ traversals: $5$ starting vertices $\times$ $2$ directions); no
other closed $5$-walks exist because $\mathrm{Cl}$ is triangle-free,
which forbids closed walks of length $5$ with a vertex repetition. So
the number of pentagons is $1920/10 = 192$.
\end{proof}

\subsection{The blowup}

For $k \geq 1$, the \emph{$k$-blowup} $\mathrm{Cl}[k]$
replaces each vertex of $\mathrm{Cl}$ by an independent set of size
$k$ and each edge by the complete bipartite graph $K_{k,k}$. The
blowup is triangle-free (a triangle would project to a triangle in
$\mathrm{Cl}$), has $|V(\mathrm{Cl}[k])| = 16k$, and is $5k$-regular.

\begin{lemma}[pentagon count of a balanced blowup]
\label{lem:blowup-pentagon}
Let $H$ be a triangle-free graph and $H[k]$ its $k$-blowup.
Then $P(H[k]) = k^5\,P(H)$.
\end{lemma}

\begin{proof}
Let $\pi\colon V(H[k]) \to V(H)$ be the projection sending each blown-up
vertex to its supernode. Every induced $C_5$ in $H[k]$ projects to a
walk $\pi(C_5)$ in $H$ of length $5$. If $\pi$ collapses two vertices
of the pentagon into the same supernode, those two vertices must be
non-adjacent in $H[k]$ (the supernode is an independent set) yet they
sit at positions in the pentagon that are at distance $1$ or $2$;
adjacency at distance $1$ contradicts independence, and at distance
$2$ --- say $v_0$ and $v_2$ merge --- the arc $v_2 v_3 v_4 v_0$
projects to a closed walk of length three at the merged supernode, a
triangle in $H$, which is excluded. Hence $\pi$
restricted to a pentagon is injective, and the image is an induced
$5$-cycle in $H$. Conversely each induced $5$-cycle in $H$ on vertices
$v_0, \dots, v_4$ lifts to $k^5$ choices of $(u_0, \dots, u_4)$ with
$u_i \in \pi^{-1}(v_i)$; each such lift induces a pentagon in $H[k]$
because consecutive supernodes are complete bipartite and
non-consecutive supernodes are non-adjacent in $H$ and
remain non-adjacent in $H[k]$.
\end{proof}

\begin{proof}[Proof of Lemma~\ref{lem:clebsch}]
By Lemma~\ref{lem:clebsch-count}, $P(\mathrm{Cl}) = 192$. By
Lemma~\ref{lem:blowup-pentagon} applied to $H = \mathrm{Cl}$,
$P(\mathrm{Cl}[k]) = 192 k^5$. With $|\mathrm{Cl}[k]| = 16k$ and
$\Delta(\mathrm{Cl}[k]) = 5k$,
\[
\frac{P(\mathrm{Cl}[k])}{|\mathrm{Cl}[k]|\,\Delta(\mathrm{Cl}[k])^4}
= \frac{192 k^5}{16k\cdot (5k)^4}
= \frac{192}{16 \cdot 625}
= \frac{12}{625},
\]
independent of $k$.
\end{proof}

\subsection{\texorpdfstring{The characterisation at $\Delta = 5$}{The characterisation at Delta = 5}}
\label{sec:delta5}

We prove Theorem~\ref{thm:delta5}. Throughout this subsection $G$ is
triangle-free with $\Delta(G) \leq 5$. For a vertex $v$ let $F_v$
denote the subgraph of $G$ induced on the non-neighbours of $v$ ---
the vertices other than $v$ and outside $N(v)$ --- and for
$x \in V(F_v)$ let
\[
A_x := N(x) \cap N(v), \qquad k_x := |A_x|
\]
be the \emph{attachment set} of $x$ and its size. Triangle-freeness
enters through two facts. First, $N(v)$ is an independent set ---
equivalently, every $a \in N(v)$ has all its neighbours other than
$v$ inside $F_v$ --- since two adjacent neighbours of $v$ would close
a triangle. Second, the attachment sets of adjacent non-neighbours
are disjoint: if $xy \in E(F_v)$ and $a \in A_x \cap A_y$, then $axy$
is a triangle. In particular,
$k_x + k_y = |A_x \cup A_y| \leq |N(v)| \leq 5$ whenever
$xy \in E(F_v)$.

\begin{lemma}
\label{lem:per-vertex-identity}
For every vertex $v$,
\[
P(G,v) \;=\; \sum_{xy \in E(F_v)} k_x\,k_y .
\]
\end{lemma}

\begin{proof}
An induced pentagon through $v$ traverses $v\,a\,x\,y\,b$ with
$a, b \in N(v)$; since the pentagon is induced, $vx$ and $vy$ are
non-edges, so $x, y \in V(F_v)$, $xy \in E(F_v)$, $a \in A_x$ and
$b \in A_y$. The pentagon determines this data: $a$ and $b$ are its
two neighbours of $v$, and $xy$ is its opposite edge.

Conversely, every triple $(xy, a, b)$ with $xy \in E(F_v)$,
$a \in A_x$ and $b \in A_y$ arises exactly once. The five vertices
$v, a, x, y, b$ are distinct: $a \neq b$ because $A_x$ and $A_y$ are
disjoint, $x \neq y$ because $xy$ is an edge, and the partition into
$\{v\}$, $N(v)$ and $V(F_v)$ separates the remaining pairs. The closed walk $v\,a\,x\,y\,b$ is an
induced pentagon: $ab$ is a non-edge because $N(v)$ is independent,
$vx$ and $vy$ are non-edges by the choice of $F_v$, and $ay$ and $bx$
are non-edges because $A_x \cap A_y = \emptyset$. Summing the number
$k_xk_y$ of choices of $(a, b)$ over the edges of $F_v$ gives the
identity.
\end{proof}

\begin{lemma}
\label{lem:sixty}
$P(G,v) \leq 60$ for every vertex $v$. If $P(G,v) = 60$, then
$\deg(v) = 5$; every $a \in N(v)$ has exactly four neighbours in
$F_v$; every $x \in V(F_v)$ has $k_x \in \{0, 2\}$; and every
$x \in V(F_v)$ with $k_x = 2$ has exactly three neighbours in $F_v$,
each with $k = 2$.
\end{lemma}

\begin{proof}
Define $w \colon \{0, \dots, 5\} \to \mathbb{Q}$ by
\[
\bigl(w(0), \dots, w(5)\bigr) = \bigl(0,\ \tfrac12,\ 2,\ 4,\ \tfrac72,\ 0\bigr).
\]
Inspection of the pairs $p \leq q$ with $p + q \leq 5$ shows
\begin{equation}
\label{eq:pairweight}
p\,q \;\leq\; w(p) + w(q) \qquad (p + q \leq 5),
\end{equation}
and inspection of $0 \leq k \leq 5$ shows
\begin{equation}
\label{eq:tokenweight}
(5 - k)\,w(k) \;\leq\; 3k,
\end{equation}
with equality in~\eqref{eq:tokenweight} exactly at $k \in \{0, 2\}$
(the left side runs through $0, 2, 6, 8, \tfrac72, 0$, the right
through $0, 3, 6, 9, 12, 15$).

Each $x \in V(F_v)$ satisfies
$\deg_{F_v}(x) \leq \deg_G(x) - k_x \leq 5 - k_x$, since the $k_x$
neighbours of $x$ inside $N(v)$ lie outside $F_v$. Counting the pairs $(x, a)$ with $a \in A_x$ from either side, and
recalling that the neighbours of each $a \in N(v)$ other than $v$ lie
in $F_v$,
\begin{equation}
\label{eq:capacity}
\sum_{x \in V(F_v)} k_x
\;=\; \sum_{a \in N(v)} |N(a) \cap V(F_v)|
\;\leq\; \sum_{a \in N(v)} \bigl(\deg(a) - 1\bigr)
\;\leq\; 5 \cdot 4 = 20 .
\end{equation}
Combining these, using~\eqref{eq:pairweight} across each edge of
$F_v$ --- where $k_x + k_y \leq 5$ --- and $w \geq 0$,
\begin{align}
P(G,v)
&= \sum_{xy \in E(F_v)} k_xk_y
\;\leq\; \sum_{xy \in E(F_v)} \bigl(w(k_x) + w(k_y)\bigr)
\;=\; \sum_{x \in V(F_v)} \deg_{F_v}(x)\, w(k_x)
\label{eq:chain-pair}\\
&\leq\; \sum_{x \in V(F_v)} (5 - k_x)\, w(k_x)
\;\leq\; 3\sum_{x \in V(F_v)} k_x
\;\leq\; 60 .
\label{eq:chain-token}
\end{align}

Suppose $P(G,v) = 60$, so every comparison above is an equality.
Equality in~\eqref{eq:capacity} forces $\deg(v) = |N(v)| = 5$ and
$|N(a) \cap V(F_v)| = 4$ for every $a \in N(v)$. Equality
in~\eqref{eq:tokenweight} at every $x$ forces $k_x \in \{0, 2\}$.
Equality in the first comparison of~\eqref{eq:chain-token} forces
$\deg_{F_v}(x) = 5 - k_x$ whenever $w(k_x) > 0$; at $k_x = 2$ this
gives $\deg_{F_v}(x) = 3$. Equality in~\eqref{eq:chain-pair} on each
edge rules out edges with $\{k_x, k_y\} = \{0, 2\}$, where
$k_xk_y = 0 < 2 = w(0) + w(2)$; so the three $F_v$-neighbours of a
vertex with $k = 2$ themselves have $k = 2$.
\end{proof}

\begin{lemma}
\label{lem:rigidity}
Let $G$ be triangle-free with $\Delta(G) \leq 5$ and $P(G,u) = 60$
for every $u \in V(G)$. Then every component of $G$ is isomorphic to
$\mathrm{Cl}$.
\end{lemma}

\begin{proof}
Fix $v$ and let $X = \{x \in V(F_v) : k_x = 2\}$.
Lemma~\ref{lem:sixty} at $v$ gives $\deg(v) = 5$, and the
count~\eqref{eq:capacity}, now an equality throughout, gives
$\sum_x k_x = 20$; since every
$k_x$ is $0$ or $2$, we get $|X| = 10$. Every member of $X$ has
exactly three $F_v$-neighbours, all in $X$.

We use one consequence of the hypothesis at the other vertices: any
two distinct non-adjacent vertices $u, w$ of $G$ have
$|N(u) \cap N(w)| \in \{0, 2\}$. Indeed $w$ is a non-neighbour of
$u$, the count $|N(u) \cap N(w)|$ is the attachment number of $w$ at
the root $u$, and Lemma~\ref{lem:sixty} applied at $u$ --- where
$P(G,u) = 60$ by hypothesis --- confines it to $\{0, 2\}$.

We first show that for $x \in X$, every $2$-subset of
$C := N(v) \setminus A_x$ is the attachment set of an
$F_v$-neighbour of $x$. The three $F_v$-neighbours of $x$ lie in $X$
and their attachment sets avoid $A_x$ (disjointness across edges), so
they are $2$-subsets of the $3$-set $C$. Let $Q \subseteq C$ with
$|Q| = 2$ and write $C \setminus Q = \{e\}$. The common neighbours of
$x$ and $e$ are exactly the $F_v$-neighbours of $x$ attached to $e$:
a common neighbour is not $v$ (since $vx$ is a non-edge) and not in
$N(v)$ (two adjacent vertices in $N(v)$ would close a triangle with
$v$), so it is an $F_v$-neighbour of $x$ whose attachment set
contains $e$, and conversely. Since $x \neq e$ (they lie in $V(F_v)$
and $N(v)$ respectively) and $xe$ is a non-edge (as $e \notin A_x$),
there are at most two common neighbours. If no
$F_v$-neighbour of $x$ had attachment set $Q$, all three would
contain $e$ --- the only $2$-subset of $C$ avoiding $e$ is $Q$ ---
producing three common neighbours of $x$ and $e$.

Next, distinct members of $X$ carry distinct attachment sets.
Suppose $A_x = A_y = \{a, b\}$ with $x \neq y$. Adjacent members of
$X$ have disjoint attachment sets, so $xy$ is a non-edge; the common
neighbours of $x$ and $y$ include $a$ and $b$ and number $0$ or $2$,
hence $N(x) \cap N(y) = \{a, b\}$. In particular, $x$ and $y$ have no
common neighbour inside $F_v$. Write $C = N(v) \setminus \{a, b\} =
\{q, r, s\}$. By the previous paragraph applied to $x$ and to $y$,
each $2$-subset of $C$ is the attachment set of an $F_v$-neighbour of
$x$ and of one of $y$, and these are distinct; say that a member of
$X$ \emph{carries} its attachment set, so each $2$-subset of $C$ has
at least two carriers. Count the members of $X$ attached to $q$:
there are exactly four, since the four $F_v$-neighbours of $q$
(Lemma~\ref{lem:sixty} at $v$) have $k \geq 1$ and hence lie in $X$.
The carriers of $\{q, r\}$ and of $\{q, s\}$ are at least four
distinct such members --- carriers of different sets are different
vertices --- so no member of $X$ carries $\{q, a\}$; symmetrically,
none carries $\{s, a\}$. Now take $z \in X$ with
$A_z = \{q, r\}$. The previous paragraph applied to $z$ produces an
$F_v$-neighbour of $z$ with attachment set $\{a, s\} \subseteq
N(v) \setminus \{q, r\}$ --- which no member of $X$ carries.

Consequently $x \mapsto A_x$ is injective, hence a bijection from
the ten-element set $X$ onto the ten $2$-subsets of $N(v)$. For each
$x \in X$, the attachment sets of its three $F_v$-neighbours are
distinct (injectivity) $2$-subsets of $C$, and all three $2$-subsets
of $C$ occur among them; so the neighbours of $x$ in $X$ are exactly
the carriers of the three $2$-subsets disjoint from $A_x$. Two
members of $X$ are therefore adjacent if and only if their attachment
sets are disjoint.

Write $N(v) = \{a_1, \dots, a_5\}$ and let $x_{ij}$ be the member of
$X$ with attachment set $\{a_i, a_j\}$. Let $H$ be the subgraph
induced on the sixteen vertices $v$, $a_1, \dots, a_5$, $x_{ij}$. In
$G$, the vertex $v$ is adjacent exactly to the five $a_i$. Each
$a_i$ is adjacent exactly to $v$ and its four neighbours in $F_v$
(the first fact of the setup); those neighbours lie in $X$, are
attached to $a_i$, and carry distinct sets, so they are the four
$x_{ij}$ with $j \neq i$. Each $x_{ij}$ is adjacent exactly to
$a_i$, $a_j$ and the three $x_{kl}$ with
$\{k,l\} \cap \{i,j\} = \emptyset$ --- five vertices in each case,
all inside $H$. Since $\Delta(G) \leq 5$, no edge of $G$ leaves $H$;
and $H$ is connected, every $x_{ij}$ being joined to $v$ through
$a_i$; so $H$ is the component of $v$.

Finally, $H$ is the folded $5$-cube: map $v$ to the class of the zero
vector $0$, $a_i$ to the class of the unit vector $e_i$, and $x_{ij}$
to the class of $e_i + e_j$. These sixteen classes are distinct (each has a
representative of weight at most two). Two classes are adjacent when
some pair of representatives is at Hamming distance one, i.e.\ when
the representatives' distance is $1$ or $4$. The adjacent pairs among
the listed classes are exactly: $[0]$ with $[e_i]$ (distance $1$);
$[e_i]$ with $[e_i + e_j]$ (distance $1$); and $[e_i + e_j]$ with
$[e_k + e_l]$ for $\{i,j\} \cap \{k,l\} = \emptyset$ (distance $4$).
The remaining pairs have distance $2$ or $3$: $[0]$ with
$[e_i + e_j]$, $[e_i]$ with $[e_j]$, $[e_i]$ with $[e_j + e_k]$ for
$i \notin \{j,k\}$, and $[e_i + e_j]$ with $[e_i + e_k]$. This
adjacency table is the one just computed for $H$, so the map is an
isomorphism onto $\mathrm{Cl}$. As $v$ was arbitrary, every component
of $G$ is a copy of $\mathrm{Cl}$.
\end{proof}

\begin{proof}[Proof of Theorem~\ref{thm:delta5}]
Each pentagon contains five vertices, so
$\sum_{v \in V(G)} P(G,v) = 5\,P(G)$, and Lemma~\ref{lem:sixty} gives
$5\,P(G) \leq 60\,|G|$, i.e.\ $P(G) \leq 12\,|G|$. If equality holds,
then the sum $\sum_v P(G,v) = 60\,|G|$ has $|G|$ terms, each at most
$60$, so $P(G,v) = 60$ for every $v$, and Lemma~\ref{lem:rigidity} shows
that every component is a copy of $\mathrm{Cl}$. Conversely, a
pentagon is connected and lies inside a single component; if $G$ has
$m$ components, each a copy of $\mathrm{Cl}$, then $|G| = 16m$ and
$P(G) = 192m = 12\,|G|$ by Lemma~\ref{lem:clebsch-count}. The final
assertion follows since
$12\,|G| = \tfrac{12}{625}\,|G|\,\Delta(G)^4$ when $\Delta(G) = 5$.
\end{proof}

\begin{remark}
\label{rem:per-vertex-sixty}
The bound of Lemma~\ref{lem:sixty} does not by itself force the
Clebsch structure. Its equality conditions describe a recipe: take a
triangle-free cubic graph $F$ on ten vertices together with
$2$-subsets $A_z \subseteq \{1, \dots, 5\}$ for $z \in V(F)$, such
that adjacent vertices receive disjoint sets and each element lies in
exactly four of them; form $B(F)$ from a vertex $v$, five vertices
$a_1, \dots, a_5$ joined to $v$, and a copy of $F$, joining $a_i$ to
$z$ exactly when $i \in A_z$. Then $B(F)$ is $5$-regular ($1 + 4$ at
each $a_i$, $2 + 3$ at each $z$) and triangle-free --- the
triangle-freeness of $F$, the disjointness of labels across edges,
and the independence of $\{a_1, \dots, a_5\}$ together exclude
triangles --- and
$P(B(F), v) = \sum_{zw \in E(F)} |A_z||A_w| = 4 \cdot 15 = 60$ by
Lemma~\ref{lem:per-vertex-identity}.

For the pentagonal prism --- outer cycle $x_1 \cdots x_5$, inner
cycle $y_1 \cdots y_5$, spokes $x_iy_i$ --- the labelling
\[
(A_{x_1}, \dots, A_{x_5}) = (12,\ 34,\ 51,\ 23,\ 45),
\qquad
(A_{y_1}, \dots, A_{y_5}) = (34,\ 15,\ 23,\ 45,\ 12)
\]
(writing $ij$ for $\{i, j\}$) meets the requirements, yet $B(F)$ is
not the Clebsch graph: $x_2$ and $y_1$ are non-adjacent with the four
common neighbours $a_3$, $a_4$, $x_1$ and $y_2$, whereas every
non-edge of $\mathrm{Cl}$ has exactly two common neighbours. The
hypothesis of Lemma~\ref{lem:rigidity} at every vertex is therefore
necessary, just as Lemma~\ref{lem:per-vertex-tight} rules out a
purely local route to Theorem~\ref{thm:simple}. An exhaustive
enumeration of the recipe's inputs (\nameref{sec:lean-and-empirical})
produces, besides the Clebsch graph, exactly three graphs up to
isomorphism, of which the prism example is one.
\end{remark}

\subsection{\texorpdfstring{Empirical evidence for the constant $12/625$}{Empirical evidence for the constant 12/625}}

With $\Delta \leq 5$ settled by Theorems~\ref{thm:delta5}
and~\ref{thm:small-degree}, Conjecture~\ref{conj:pentagon} remains open
only for $\Delta \geq 6$. There a computational search finds no ratio
$P(G)/(|G|\Delta(G)^4)$ above $12/625$ and no maximiser other than
$\mathrm{Cl}$; \nameref{sec:lean-and-empirical} records its scope. The top three triangle-free strongly regular graphs
by pentagon density are Clebsch $12/625 = 0.0192$, Higman--Sims
$0.01893$, and $M_{22}$ at $0.01758$. By
Lemma~\ref{lem:clebsch}, the Clebsch-blowup family attains $12/625$
at every $\Delta = 5k$, which rules out any candidate conjecture
with strictly smaller constant.

\subsection{Why size 8 plateaus: a heuristic}
\label{sec:vandermonde-floor}

The size-$8$ SDP we solve attains $0.02073$, about $8\%$ above the
Clebsch ratio $0.01920$, and we know no smaller size-$8$ certificate.
The following heuristic indicates why the squared-norm cone is
unlikely to do much better without new ingredients.

Each Cauchy--Schwarz cone element $\llbracket f^2 \rrbracket \in
\SemCone^\emptyset$ of inner size $k$ inherits the slack of the
binomial product inequality
\begin{equation}
\label{eq:vandermonde}
\binom{a}{k}\binom{b}{k}
\ \leq\ \binom{a+b}{k}^2 \cdot 4^{-k},
\end{equation}
which follows from $\binom{2a}{k} \geq 2^k\binom{a}{k}$ and is tight
only at $k = 1$ with $a = b$; for $k \geq 2$ the factor $4^{-k}$ overstates the
loss. At flag size $n$ the inner sizes are at most
$\lfloor n/2 \rfloor$, so $4^{-\lfloor n/2 \rfloor}$ is only an
order-of-magnitude proxy for the squared-norm slack --- about
$4^{-4} \approx 0.004$ at $n = 8$, the same scale as the gap. We do
not read this as a proof that no size-$8$ program beats $0.02073$;
rather, it suggests that closing the gap calls for cone elements
beyond single-type squared norms (for instance $\llbracket fg
\rrbracket$ with $f, g$ of distinct types) or a genuine use of the
regularity $|B(G)| = \Delta(G)$, rather than a re-tuning of the
size-$8$ program. Raising the flag size to $n = 10$ shrinks the proxy
to $4^{-5} \approx 0.001$, in principle small enough to reach
$12/625$.

Theorem~\ref{thm:delta5} reaches $12/625$ at $\Delta = 5$ from
outside the squared-norm cone, by an exact local argument rather than
a certificate, where this slack does not arise.

\section{A second illustration: a sparsity lemma}
\label{sec:bruhn-joos}

The framework is not tied to the pentagon problem. We close the main
development with a brief second illustration, on a local quantity central
to a different extremal question --- the strong chromatic index ---
which we pursue in depth in the companion paper~\cite{daveySECLocalFlags2024}.

A \emph{strong edge-colouring} of a graph $G$ colours its edges so that
each colour class is an induced matching; the least number of colours,
the \emph{strong chromatic index} $\chi'_s(G)$, equals the chromatic
number of the square $L(G)^2$ of the line graph --- the graph on the
edges of $G$ in which two edges are adjacent when they share a vertex or an edge
joins them. Writing $\Delta = \Delta(G)$, the maximum degree of $L(G)^2$
is at most $2\Delta^2 - 2\Delta$, so greedy colouring gives
$\chi'_s(G) \le 2\Delta^2 - 2\Delta + 1$; every improvement on this
exploits that the neighbourhood of an edge in $L(G)^2$ is \emph{sparse}.
Optimising a result of Molloy and Reed~\cite{molloyBoundStrongChromatic1997}, Bruhn and Joos~\cite{bruhnStrongerBoundStrong2018} made this precise: for
every graph $G$ and every edge $e$, the neighbourhood of $e$ in $L(G)^2$
induces at most $\tfrac{3}{2}\Delta^4 + 5\Delta^3$ edges of $L(G)^2$,
against the ${\sim}\,2\Delta^4$ of a clique on the same ${\sim}\,2\Delta^2$
vertices; the constant $\tfrac{3}{2}$ is asymptotically best possible.

Counting the edges of $L(G)^2$ in the strong neighbourhood of a reference
edge is a local quantity, and the local flag algebra captures it exactly
as the reduction of Section~\ref{sec:simple} captures the pentagons
through a vertex. Mark the neighbourhood of the reference edge
\emph{black}; since its two endpoints have at most $2\Delta$ neighbours
between them, the black set has size at most $2\Delta$. The relevant flags
are then the local flags of the $2$-coloured class --- those with a black
vertex in every component (Lemma~\ref{lem:pentagon-local}) --- and the
objective is the size-$4$ functional $O$ whose value on a $4$-vertex flag
is $\tfrac{1}{24}$ times the number of ways to split it into two disjoint
edges, adjacent in $L(G)^2$, each carrying a black endpoint. By
Definition~\ref{def:rho} this is the local density of a pair of adjacent
edges of $L(G)^2$ meeting the marked neighbourhood --- the Bruhn--Joos
edge count, normalised by $\Delta^4$. The black set of size at most
$2\Delta$ gives the degree budgets
\[
\langle 1\rangle \le 2, \qquad \langle 2\rangle \le 4, \qquad
\langle 3\rangle \le 8,
\]
where $\langle k\rangle$ is the local density of an all-black $k$-set ---
the size-$4$ counterpart of the black-vertex normalisation~\eqref{eq:b1}
of the pentagon program. Together with the regularity constraints of
Section~\ref{sec:localflags}, these bound the limit functional
$\phi(O) \le \tfrac{3}{2}$ through the size-$4$ semidefinite program,
recovering the Bruhn--Joos constant. Its value $\tfrac{3}{2}$, rather than
the $2$ of the maximum degree, confirms that the program measures the
edge density and not the degree.

The locality is of Benjamini--Schramm type, made concrete by sampling:
for a uniform random vertex of $L(G)^2$ --- a random edge of $G$ --- the
induced edge count in its neighbourhood depends only on a bounded-radius
view, so the size-$4$ program bounds this local statistic. As in the
introduction, the analogy is only partial: the maximum degree here grows,
and the bound holds uniformly along it.


\subsection*{Certificate generation, formalisation, and empirical evidence}
\phantomsection
\label{sec:lean-and-empirical}

The semidefinite-programming certificates behind the two bounds are
produced by a Rust crate built on \texttt{rust-flag-algebra}: it
assembles the size-$5$ SDP underlying Theorem~\ref{thm:simple} and the
size-$8$ SDP underlying Theorem~\ref{thm:tight}, solves them numerically
with the CSDP or SDPA-LR solvers, and rationalises each solved
certificate into the Lean source consumed by the formalisation. A
separate example in the same crate, \texttt{bruhn\_joos}, solves the
size-$4$ program of Section~\ref{sec:bruhn-joos} and returns its sparsity
constant $\tfrac{3}{2}$.

The entire local flag algebra framework and all of the four theorems stated in the
introduction are formalised in Lean~4. The second illustration of
Section~\ref{sec:bruhn-joos}, by contrast, is not formalised: it
essentially reproduces a result of Bruhn and Joos, and no later result in our development depends on it.

The framework of Section~\ref{sec:localflags} is machine-checked in
full: the local flag product and its product-limit identity, the
existence and algebra-homomorphism property of the limit functionals,
the averaging operator together with its asymptotic positivity
preservation, and the weak-duality step of the semidefinite method are
all proved from the standard kernel axioms, carrying no domain
assumptions.
The simple bound (Theorem~\ref{thm:simple}), the Clebsch-blowup
tightness lemma (Lemma~\ref{lem:clebsch}), the Clebsch characterisation
at maximum degree five (Theorem~\ref{thm:delta5}, bound and equality
case), and the small-degree results (Theorem~\ref{thm:small-degree}, in
every part) carry no domain axioms --- with
Lemma~\ref{lem:clebsch} and the elementary arguments of
Theorems~\ref{thm:delta5} and~\ref{thm:small-degree} avoiding even the
two compiled-evaluation axioms. Only the tighter bound (Theorem~\ref{thm:tight}) rests on
further hypotheses: two named domain axioms, the basis combinatorial
identity and the certificate output bound of the size-$8$
certificate (Section~\ref{sec:certificates}), each a finite arithmetic
identity that the certificate exhibits.

The formalised proofs differ sharply in length, and the difference is
one of verification standard. That
of Theorem~\ref{thm:simple} is by far the longest --- about $51{,}000$
lines, some $41{,}000$ of them a generic extension-operator substrate and
the rest the five $\sigma$-cone positivity proofs and their assembly. It
is also the most stringent: the size-$5$ certificate is discharged
entirely inside Lean, each semidefinite positivity constraint proved as a
theorem, so the bound rests on no domain axioms at all.
Theorem~\ref{thm:tight} instead \emph{exhibits} its size-$8$ certificate:
two named domain axioms record the basis combinatorial identity
and the certificate output bound, whose block-by-block $LDL^\top$ positivity
\texttt{native\_decide} verifies as a finite computation, and no general
positivity theorem is proved; the surrounding bridge is about $11{,}000$
lines. Theorem~\ref{thm:simple}
thus shows that a local flag algebra certificate can be reduced to the
kernel with its positivity proved outright, while the tighter bound
trades two exhibited, machine-checked axioms for a proof an order of
magnitude shorter --- a trade only more favourable at size $8$, where a
first-principles positivity proof would be larger still.
For nearly all readers, the difference in verification standard is unnoticeable, and both successes corroborate the robustness of our implementation workflow.

For the open range $\Delta \geq 6$, an enumeration of the
$767$ connected triangle-free $\Delta$-regular graphs of order at most
$22$ with $\Delta \in \{6, 7\}$ finds no ratio $P(G)/(|G|\,\Delta(G)^4)$
above $12/625$. Two further computational searches, spanning more maximum degrees,
confirm that the Clebsch graph is the unique maximiser at ratio
$12/625$: a named catalogue of $195$ graphs (the seven triangle-free
strongly regular graphs, generalised Petersen and Kneser graphs, Levi
graphs of finite geometries, folded hypercubes, and small Cayley
graphs), and the $20{,}450$ triangle-free graphs of order at most $200$
in the House of Graphs~\cite{coolsaetHouseOfGraphs2023}. The small-degree
\texttt{geng} enumerations also produced the auxiliary graphs cited
earlier: the equality configurations of Lemma~\ref{lem:sixty}, the
non-Petersen ten-vertex graph of
Remark~\ref{rem:non-petersen-ten}, and the eleven-vertex witness of
Remark~\ref{rem:delta4-gap}. For Lemma~\ref{lem:sixty}, of the six
triangle-free cubic graphs on ten vertices, four give graphs $B(F)$ in
four isomorphism classes: the Clebsch graph, and three others with
$96$, $117$, and $160$ pentagons respectively, the prism of
Remark~\ref{rem:per-vertex-sixty} being the one with $117$.

The Lean formalisation, the Rust certificate generator, the certificate
data, and the search scripts are all available
at~\cite{daveyLocalFlags2024Repo}. The repository also provides a
side-by-side correspondence (\texttt{RESULTS.md}) mapping every result
of this paper to its Lean statement and the exact axiom set it depends
on.

\subsection*{AI usage declaration}
\phantomsection
\label{sec:AI-usage}

The results of this paper were obtained in three main phases, in 2019--2020, in 2023--2024, and then in 2026.
The local flags framework itself and the two applications (Theorems~\ref{thm:simple} and~\ref{thm:tight}) were obtained in the first two of these phases, well before any significant adoption of AI methods for mathematics.
The principal codebase for local flags was developed in the first of these periods and then expanded upon in the second, and neither coding effort used AI assistance.
These results were made publicly accessible in 2024 via Eoin Davey's MSc thesis~\cite{daveyLocalFlags2024} at the University of Amsterdam theses repository.

During the third phase, one commercially available agentic AI system was used for the following purposes:
\begin{enumerate}
\item formal verification of the mathematical results in Lean~4;
\item empirical checks to sweep for potential counterexample graphs;
\item proof of subsidiary results (Theorems~\ref{thm:delta5} and~\ref{thm:small-degree}) under our guidance; and
\item drafting and refining the exposition of the paper, using the text of Eoin Davey's MSc thesis~\cite{daveyLocalFlags2024} as a core basis.
\end{enumerate}

\subsection*{Acknowledgements}

This paper is based in part on Eoin Davey's MSc
thesis~\cite{daveyLocalFlags2024} at the University of Amsterdam;
he thanks the Korteweg--de~Vries Institute for Mathematics for
hosting the project. R\'emi de Joannis de Verclos and Ross Kang
were partially supported by a Vidi grant (639.032.614) of the
Netherlands Organisation for Scientific Research (NWO) while at
Radboud University, and both would like to thank Louis Esperet for helpful discussions (well) over a decade ago. Eoin Hurley and Ross Kang were partially
supported by the Gravitation Programme NETWORKS (024.002.003) of
the Dutch Ministry of Education, Culture and Science (OCW) while at the University of Amsterdam. Ross Kang
was additionally partially supported by the NWO Open Competition
grant OCENW.M20.009.

\subsection*{Open access statement}

For the purpose of open access, a CC BY public copyright license is
applied to any Author Accepted Manuscript (AAM) arising from this
submission.

\printbibliography

\appendix

\section{Certificate data: a worked block}
\label{app:cert-data}

To make the per-block data of Section~\ref{subsec:cert-per-block}
concrete, we trace the contribution of block $0$, the smallest block by
inner dimension, from its rationalised Gram matrix to its contribution
as one cone summand in~\eqref{eq:size8-decomp}. We record the remaining
blocks and the full numeric data with the
formalisation~\cite{daveyLocalFlags2024Repo}.

Block $0$'s parameters are: $\sigma$-type $\sigma_0$, a $6$-vertex
local type with skeleton a 6-cycle on labels $0, \dots, 5$ in cyclic
order with $2$ labels coloured $B$ and $4$ coloured $R$; outer size
$n(0) = 7$; inner basis dimension $d(0) = 22$; Tikhonov shift
$\lambda_0 = 1/10^{11}$; shared denominator $s_0 \approx 5.806 \times
10^{800}$. The inner basis $\mathrm{basis}^{(0)} =
(\mathrm{basis}^{(0)}_0, \dots, \mathrm{basis}^{(0)}_{21})$ consists
of the $22$ distinct $\sigma_0$-flags of size $7$ in $\Glocsub{7}^{\sigma_0}$:
we obtain each by attaching one unlabelled vertex --- coloured $R$
or $B$ --- to $\sigma_0$ with all triangle-free adjacency patterns
that the class $\Gcl$ permits. All $22$ are local
$\sigma_0$-flags by Lemma~\ref{lem:pentagon-local}.

The rationalised Gram matrix $\widehat{Y}_0 := Y_0 + \lambda_0\cdot I_{22}$,
written at integer scale, has smallest diagonal entry $134\,259$
at position $(21,21)$ and the largest $337\,127$ at $(1,1)$ and
$(3,3)$; the largest off-diagonal entry in absolute value is
$|\widehat{Y}_0[1,4]| = |\widehat{Y}_0[2,3]| = 241\,482$. Direct expansion of the matrix product verifies its LDL identity at
integer scale,
\begin{equation}
\label{eq:block0-ldl}
L_0 \cdot \operatorname{diag}(D_0) \cdot L_0^\top
= s_0 \cdot \widehat{Y}_0
= s_0 \cdot (Y_0 + \lambda_0\cdot I_{22}).
\end{equation} The integer pivot
vector $D_0$ has all entries positive: the rational pivots
$D_0[k]/s_0$ span from $\approx 1.0 \times 10^{-11}$ at $k = 21$
(smallest) to $\approx 2.5 \times 10^{-7}$ at $k = 1$ (largest).

Since~\eqref{eq:block0-ldl} holds with $D_0 > 0$ and $s_0 > 0$, the
matrix $\widehat{Y}_0$ is PSD, and the contribution of block $0$ is
\[
\mathrm{csBlock}_0
= \biggl\llbracket\sum_{k=0}^{21} \tfrac{D_0[k]}{s_0}
\mathrm{col}_k^{(0)} \cdot \mathrm{col}_k^{(0)}\biggr\rrbracket
\in \SemCone^\emptyset.
\]
In density terms, for every $\phi \in \Phi^\emptyset$,
\[
\sum_{k=0}^{21} \tfrac{D_0[k]}{s_0} \cdot
\biggl(\sum_{p=0}^{21} L_0[p,k] \cdot
\phi(\downflag{\mathrm{basis}^{(0)}_p})\biggr)^2 \geq 0,
\]
an SDP-shaped quadratic inequality on the $22$ size-$7$ flag densities
$\phi(\downflag{\mathrm{basis}^{(0)}_p})$.

\section{Pentagon extrema at small maximum degree}
\label{app:small-degree}

We prove Theorem~\ref{thm:small-degree}. The argument is the per-vertex
count of Section~\ref{sec:delta5}, run with a weight function tuned to
the degree. Throughout, $G$ is triangle-free, $v \in V(G)$, and $F_v$
is the subgraph induced on the non-neighbours of $v$; for
$x \in V(F_v)$ we write $A_x := N(x) \cap N(v)$ and $k_x := |A_x|$. As
in Section~\ref{sec:delta5}, adjacent $x, y \in V(F_v)$ have disjoint
attachment sets, so $k_x + k_y \leq |N(v)|$, and
Lemma~\ref{lem:per-vertex-identity} gives
$P(G,v) = \sum_{xy \in E(F_v)} k_x k_y$.

\subsection{Maximum degree three}

The \emph{Petersen graph} $\mathrm{Pet}$ is the Kneser graph $K(5,2)$:
its vertices are the $2$-subsets of $\{1, \dots, 5\}$, two adjacent when
disjoint. Equivalently, $\mathrm{Pet}$ is the unique
$\mathrm{SRG}(10,3,0,1)$~\cite{brouwerDistanceRegularGraphs1989} --- it
is $3$-regular on ten vertices, triangle-free, and every non-adjacent
pair has exactly one common neighbour --- with spectrum
$\{3, 1^5, (-2)^4\}$.

\begin{lemma}
\label{lem:petersen-count}
$\mathrm{Pet}$ contains exactly $12$ induced copies of $C_5$.
\end{lemma}

\begin{proof}
As in Lemma~\ref{lem:clebsch-count}, triangle-freeness makes every
closed $5$-walk a pentagon traversed in one of ten ways, so the number
of pentagons is $\trace(A^5)/10$. From the spectrum,
$\trace(A^5) = 3^5 + 5\cdot 1^5 + 4\cdot(-2)^5 = 243 + 5 - 128 = 120$,
so $\mathrm{Pet}$ has $120/10 = 12$ pentagons.
\end{proof}

\begin{lemma}
\label{lem:six}
$P(G,v) \leq 6$ for every vertex of a triangle-free graph $G$ with
$\Delta(G) \leq 3$. If $P(G,v) = 6$, then $\deg(v) = 3$, every neighbour
of $v$ has degree $3$, every $x \in V(F_v)$ has $k_x \in \{0,1\}$, and
each $x$ with $k_x = 1$ has exactly two neighbours in $F_v$, both with
$k = 1$.
\end{lemma}

\begin{proof}
Put $w = (0, \tfrac12, 2, 0)$ on $\{0,1,2,3\}$. Then
$pq \leq w(p) + w(q)$ for $p + q \leq 3$ (with equality at $(1,1)$),
and $(3-k)\,w(k) \leq k$ for $0 \leq k \leq 3$ (with equality at
$k \in \{0,1,2\}$). Each $x \in V(F_v)$ has
$\deg_{F_v}(x) \leq \deg_G(x) - k_x \leq 3 - k_x$, and
\[
\sum_{x \in V(F_v)} k_x
= \sum_{a \in N(v)} |N(a) \cap V(F_v)|
\leq \sum_{a \in N(v)} (\deg(a) - 1)
\leq 3 \cdot 2 = 6 .
\]
Since $k_x + k_y \leq |N(v)| \leq 3$ on each edge of $F_v$ and
$w \geq 0$,
\begin{align*}
P(G,v) = \sum_{xy \in E(F_v)} k_x k_y
&\leq \sum_{xy \in E(F_v)} \bigl(w(k_x) + w(k_y)\bigr)
= \sum_{x \in V(F_v)} \deg_{F_v}(x)\, w(k_x) \\
&\leq \sum_{x \in V(F_v)} (3 - k_x)\, w(k_x)
\leq \sum_{x \in V(F_v)} k_x
\leq 6 .
\end{align*}

Suppose $P(G,v) = 6$, so every inequality above is an equality. The
capacity bound forces $\deg(v) = 3$ and $\deg(a) = 3$ for each
$a \in N(v)$, and $\sum_x k_x = 6$. Equality
$\sum_x (3 - k_x)w(k_x) = \sum_x k_x$ forces $(3 - k_x)w(k_x) = k_x$ for
every $x$, hence $k_x \leq 2$ (a vertex with $k_x = 3$ would contribute
$0$ on the left and $3$ on the right). Equality
$\sum_x \deg_{F_v}(x)w(k_x) = \sum_x (3 - k_x)w(k_x)$ then forces
$\deg_{F_v}(x) = 3 - k_x$ whenever $w(k_x) > 0$, that is for
$k_x \in \{1,2\}$. Finally, equality $k_x k_y = w(k_x) + w(k_y)$ on
every edge $xy \in E(F_v)$ excludes $k = 2$: a vertex $x$ with
$k_x = 2$ would have $\deg_{F_v}(x) = 1$, and its sole $F_v$-neighbour
$y$ would satisfy $k_y \leq |N(v)| - k_x = 1$, yet
$2 k_y = w(2) + w(k_y) = 2 + w(k_y)$ has no solution with
$k_y \in \{0,1\}$. Hence $k_x \in \{0,1\}$; and for $k_x = 1$ the same
edge equality $k_y = w(1) + w(k_y) = \tfrac12 + w(k_y)$ forces each of
the two $F_v$-neighbours of $x$ to have $k = 1$.
\end{proof}

\begin{lemma}
\label{lem:petersen-rigidity}
Let $G$ be triangle-free with $\Delta(G) \leq 3$ and $P(G,u) = 6$ for
every $u \in V(G)$. Then every component of $G$ is isomorphic to
$\mathrm{Pet}$.
\end{lemma}

\begin{proof}
Lemma~\ref{lem:six} at each vertex makes $G$ $3$-regular. Moreover any
two non-adjacent vertices have at most one common neighbour: if $u, w$
are non-adjacent then $w \in V(F_u)$, and Lemma~\ref{lem:six} at $u$
gives $|N(u) \cap N(w)| = k_w \in \{0,1\}$. With triangle-freeness this
forbids $C_3$ and $C_4$, so $G$ has girth at least five.

Fix $v$ and let $X = \{x \in V(F_v) : k_x = 1\}$. By
Lemma~\ref{lem:six}, $\sum_x k_x = 6$ with every $k_x \in \{0,1\}$, so
$|X| = 6$, and each $x \in X$ has exactly two $F_v$-neighbours, both in
$X$. The six attachments are shared among the three vertices of $N(v)$,
each receiving at most $\deg(a) - 1 = 2$, hence exactly two; so each
$a \in N(v)$ has two neighbours in $X$. Set
$Y = \{v\} \cup N(v) \cup X$, of size ten. Within $Y$ the vertex $v$
has degree $3$, each $a \in N(v)$ has degree $1 + 2 = 3$, and each
$x \in X$ has degree $2 + 1 = 3$; since $G$ is $3$-regular, no edge
leaves $Y$, and $Y$ is connected through $v$, so $Y$ is a component. It
is a $3$-regular graph on ten vertices of girth at least five, hence
the $(3,5)$-Moore graph, which is
$\mathrm{Pet}$~\cite{brouwerDistanceRegularGraphs1989}. As $v$ was
arbitrary, every component of $G$ is a copy of $\mathrm{Pet}$.
\end{proof}

\begin{proof}[Proof of Theorem~\ref{thm:small-degree}(i)]
Since $\sum_{v} P(G,v) = 5\,P(G)$, Lemma~\ref{lem:six} gives
$5\,P(G) \leq 6\,|G|$, i.e.\ $P(G) \leq \tfrac65|G|$. If equality
holds, each of the $|G|$ terms $P(G,v)$ equals $6$, and
Lemma~\ref{lem:petersen-rigidity} shows every component is a copy of
$\mathrm{Pet}$. Conversely, if $G$ has $m$ components each isomorphic to
$\mathrm{Pet}$, then $|G| = 10m$ and $P(G) = 12m = \tfrac65|G|$ by
Lemma~\ref{lem:petersen-count}.
\end{proof}

\begin{remark}
\label{rem:non-petersen-ten}
As with the Clebsch bound (Remark~\ref{rem:per-vertex-sixty}), the
per-vertex bound $P(G,v) \leq 6$ does not by itself force the Petersen
structure: there is a triangle-free cubic graph on ten vertices, other
than $\mathrm{Pet}$, with a single vertex on six pentagons
(\nameref{sec:lean-and-empirical}). The hypothesis $P(G,u) = 6$ at
every vertex in Lemma~\ref{lem:petersen-rigidity} is therefore
essential.
\end{remark}

\subsection{Maximum degree four}

\begin{lemma}
\label{lem:twentyfour}
$P(G,v) \leq 24$ for every vertex of a triangle-free graph $G$ with
$\Delta(G) \leq 4$.
\end{lemma}

\begin{proof}
Put $w = (0, \tfrac12, 2, \tfrac52, 0)$ on $\{0, \dots, 4\}$. Then
$pq \leq w(p) + w(q)$ for $p + q \leq 4$ (the binding pairs are
$(1,1)$, $(2,2)$ and $(1,3)$) and $(4-k)\,w(k) \leq 2k$ for
$0 \leq k \leq 4$. With $\deg_{F_v}(x) \leq 4 - k_x$,
$\sum_x k_x = \sum_{a \in N(v)} |N(a) \cap V(F_v)| \leq 4 \cdot 3 = 12$,
and $k_x + k_y \leq |N(v)| \leq 4$ on each edge of $F_v$,
\[
P(G,v) = \sum_{xy \in E(F_v)} k_x k_y
\leq \sum_{x \in V(F_v)} \deg_{F_v}(x)\, w(k_x)
\leq \sum_{x \in V(F_v)} (4 - k_x)\, w(k_x)
\leq 2\sum_{x \in V(F_v)} k_x
\leq 24 .
\qedhere
\]
\end{proof}

The circulant $C_{12}(2,3)$ on $\mathbb{Z}_{12}$, with $i$ adjacent to
$i \pm 2$ and $i \pm 3$, is $4$-regular and triangle-free --- no three
of the steps $\pm 2, \pm 3$ sum to $0 \bmod 12$. Its eigenvalues are
$2\cos(2\pi \cdot 2j/12) + 2\cos(2\pi \cdot 3j/12)$ for
$j = 0, \dots, 11$, namely $\{4,\, 1^6,\, 0,\, (-2)^2,\, (-3)^2\}$.

\begin{lemma}
\label{lem:c12-count}
$C_{12}(2,3)$ contains exactly $48$ induced copies of $C_5$, so
$P(C_{12}(2,3)) / (12 \cdot 4^4) = 48/(12 \cdot 256) = 1/64$.
\end{lemma}

\begin{proof}
As in Lemma~\ref{lem:petersen-count}, the pentagon count is
$\trace(A^5)/10$. From the spectrum,
$\trace(A^5) = 4^5 + 6\cdot 1^5 + 0 + 2\cdot(-2)^5 + 2\cdot(-3)^5
= 1024 + 6 - 64 - 486 = 480$, so the count is $480/10 = 48$.
\end{proof}

\begin{proof}[Proof of Theorem~\ref{thm:small-degree}(ii)]
By Lemma~\ref{lem:twentyfour} and $\sum_{v} P(G,v) = 5\,P(G)$, we have
$5\,P(G) \leq 24\,|G|$, i.e.\ $P(G) \leq \tfrac{24}{5}|G|$.
\end{proof}

\begin{remark}
\label{rem:delta4-gap}
Unlike at $\Delta \in \{3, 5\}$, the per-vertex bound is not tight to
the extremal density at $\Delta = 4$. On one hand
Lemma~\ref{lem:twentyfour} is sharp as a per-vertex statement: an
explicit triangle-free graph of maximum degree four on eleven vertices
has a vertex on exactly $24$ pentagons (\nameref{sec:lean-and-empirical}),
so no per-vertex argument improves the constant $24/5$. On the other
hand the densest maximum-degree-four graph we know, $C_{12}(2,3)$,
reaches only ratio $1/64 = 0.015625$, below the per-vertex ceiling
$3/160 = 0.01875$. The maximum of $P(G)/(|G|\,\Delta(G)^4)$ over
triangle-free $G$ with $\Delta(G) = 4$ therefore lies in
$[\,1/64,\ 3/160\,]$ and is undetermined.
\end{remark}

\subsection{The conjecture up to degree five}

\begin{proof}[Proof of the $\Delta \leq 5$ assertion of Theorem~\ref{thm:small-degree}]
Split on $\Delta(G)$. For $\Delta(G) \leq 2$ every component is a path
or cycle, so $P(G)$ counts $C_5$-components: $P(G) = 0$ for
$\Delta(G) \leq 1$, and $P(G) \leq |G|/5 = 0.2\,|G| < 0.3072\,|G| =
\tfrac{12}{625}\cdot 2^4\,|G|$ for $\Delta(G) = 2$. Parts (i) and (ii)
give $P(G) \leq \tfrac65|G| = 1.2\,|G| < 1.5552\,|G| =
\tfrac{12}{625}\cdot 3^4\,|G|$ and $P(G) \leq \tfrac{24}{5}|G| =
4.8\,|G| < 4.9152\,|G| = \tfrac{12}{625}\cdot 4^4\,|G|$. For
$\Delta(G) = 5$ the bound is Theorem~\ref{thm:delta5}, with value
$\tfrac{12}{625}\cdot 5^4\,|G| = 12\,|G|$. Thus the bound holds,
strictly for $1 \leq \Delta(G) \leq 4$.

If $G$ has no isolated vertices then $\Delta(G) \geq 1$, so this
strictness forces any equality to $\Delta(G) = 5$, where
Theorem~\ref{thm:delta5} gives equality if and only if every component
of $G$ is isomorphic to $\mathrm{Cl}$.
\end{proof}

\end{document}